\documentclass[letterpaper, 10pt]{article} 
\usepackage[margin=1in]{geometry} 

\usepackage{graphicx}
\usepackage{graphics} 
\usepackage{epsfig} 
\usepackage{amsmath} 
\usepackage{amssymb}  
\usepackage{amsthm}

\usepackage{bm}
\usepackage{amsfonts}
\usepackage{amscd}
\usepackage{hyperref}
\usepackage{pstricks}
\usepackage{latexsym}
\usepackage{enumerate}
\usepackage{subcaption}
\usepackage{framed,fancybox}
\usepackage{bbm}
\usepackage{multirow}
\usepackage{threeparttable}
\usepackage{rotating}
\usepackage{stfloats}
\usepackage{color,soul}
\usepackage[version=4]{mhchem}
\usepackage{siunitx}
\usepackage{cite}
\usepackage{float}
\usepackage{afterpage}
\usepackage{listings}
\usepackage{xcolor}
\usepackage{amsthm}

\newcommand{\m}[1]{{\bf{#1}}}
\newcommand{\g}[1]{\boldsymbol #1}
\newcommand{\bb}[1]{\mathbb #1}
\newcommand{\C}[1]{{\cal {#1}}} 
\newcommand{\mbb}[1]{\mathbb{#1}}
\newcommand{\R}{\mbb{R}}
\newcommand{\T}{^{\sf T}}    
\newcommand{\tx}[1]{\textrm{#1}}
\newcommand{\inv}{^{-1}}  
\newcommand{\intk}{^{(k)}}  
\newcommand{\intO}{^{(1)}}  
\newcommand{\intK}{^{(K)}}  
\newcommand{\ds}{\displaystyle} 
\newcommand{\sdag}{^{\dagger}} 
\newcommand{\sddag}{^{\dagger\dagger}} 
\newcommand{\dt}[1]{\textrm{d}{#1}} 
\newcommand{\ddt}[2]{\frac{\textrm{d}{#1}}{\textrm{d}{#2}}} 
\newcommand{\deldt}[2]{\frac{\partial{#1}}{\partial{#2}}} 
\newcommand{\scinot}[1]{\times 10^{#1}} %

\newtheorem{theorem}{Theorem}
\newtheorem{proposition}{Proposition}
\newtheorem{lemma}{Lemma}

\title{\LARGE \bf Integral Form of Legendre-Gauss-Lobatto \\ Collocation for Optimal Control}

\author{Gabriela Abadia-Doyle\thanks{Ph.D.~Candidate, Department of Mechanical and Aerospace Engineering. E-mail: {\tt\small gabadia97@ufl.edu}.} \\ William W. Hager\thanks{Distinguished Professor, Department of Mathematics. E-mail: {\tt\small hager@ufl.edu}.} \\ Anil V. Rao\thanks{Professor, Department of Mechanical and Aerospace Engineering. E-mail: {\tt\small anilvrao@ufl.edu}.} \vspace{12pt}\\{\em University of Florida} \\ {\em Gainesville, FL 32611}}

\date{}

\begin{document}

\maketitle

\begin{abstract}
A new method is described for solving optimal control problems using direct collocation at Legendre-Gauss-Lobatto (LGL) points. The approach of this paper employs a polynomial approximation of the right-hand side vector field of the differential equations.  Using this polynomial approximation of the right-hand side vector field, the LGL collocation method of this paper is then developed in integral form which leads to the following key outcomes. First, the first-order optimality conditions of the LGL integral form are derived, which lead to a full-rank transformed adjoint system and novel costate estimate. Next, a derivative-like form of the LGL collocation method is obtained by multiplying the system by the inverse of an appropriate full-rank block of the integration matrix. The first-order optimality conditions of the LGL derivative-like form are then derived, leading to an equivalent full-rank transformed adjoint system and secondary novel costate estimate which is related to the costate estimate of the integral form via a linear transformation. Then, it is shown that a second integral form can be constructed by including an additional noncollocated support point, but such a point is superfluous and has no impact on the solution to the nonlinear programming problem. Finally, the method is demonstrated on two benchmark problems: a one-dimensional initial value optimal control problem with an analytic solution and a time-variant orbit raising optimal control problem.
\end{abstract}

\section{Introduction}

The goal of an optimal control problem is to optimize a specified performance index subject to dynamic constraints, path constraints, and event constraints (interior point constraints and boundary conditions) \cite{Kirk2004}.  The vast majority of optimal control problems do not have analytic solutions.  Consequently, most optimal control problems must be solved numerically.  Numerical methods for solving optimal control problems fall into two broad categories: indirect methods and direct methods.  In an indirect method, the variational first-order optimality conditions are derived, leading to a Hamiltonian boundary value problem which is then solved to obtain an extremal.  In a direct method, the control or both the state and the control of the optimal control problem are parameterized, leading to a finite-dimensional approximation in the form of a nonlinear programming problem (NLP).  This NLP is then solved numerically using well-known optimization techniques \cite{Biegler2008,Gill2002}. 

References~\cite{Biral2016,Taheri2016,Mall2022} provide recently developed approaches for solving optimal control problems based on indirect methods.  These approaches are often quite challenging for several reasons.  First, deriving the first-order optimality conditions is often tedious and error-prone, and sometimes may not be possible (for example, a problem where it is necessary to interpolate tabular data).  Furthermore, indirect methods suffer from relatively small radii of convergence due to the instability of Hamiltonian systems.  Consequently, a good initial guess is often required in order to obtain a solution to the first-order optimality conditions \cite{Bryson1975}. However, determining a good initial guess can be particularly challenging due to the nonintuitive nature of the costate. 

Different from indirect methods, direct methods are often more straightforward to formulate and implement, and in many cases, they allow for a more general problem formulation. It is noted, however, that relating the solution obtained via a direct method to the continuous necessary conditions for optimality relies on the ability to obtain an accurate costate approximation  \cite{Seywald1996,Gong2008b,Garg2010,Garg2011a,Garg2011b,Darby2011sep}.  Moreover, inaccuracies in the costate approximation can lead to inaccuracies in the state and control \cite{Garg2010}.

Over the past few decades, the particular class of {\em direct collocation methods} have become popular for solving optimal control problems \cite{Betts1998a,RaoSurvey,Dahlquist2003}. In a direct collocation method, both the state and control are parameterized at specified points called {\em collocation points} or {\em nodes} and the optimal control problem is approximated as a finite-dimensional sparse NLP.  Direct collocation methods are of particular interest when solving optimal control problems \cite{Hull1997}, in part due to the typical improvement in numerical stability of implicit versus explicit simulation methods. Reference~\cite{Betts2020} categorizes these methods as either Gauss, Radau, or Lobatto, where the collocation points lie, respectively, on either the open interval $\tau\in(-1,+1)$ (Gauss), the half-open interval $\tau\in[-1,+1)$ or $\tau\in(-1,+1]$ (Radau), or the closed interval $\tau\in[-1,+1]$ (Lobatto). 

Early implementations of Lobatto collocation methods followed from relatively low-order rules from the Gauss-Lobatto family, such as the trapezoid rule and Simpson's rule \cite{Ascher1995}. Higher-order Gauss-Lobatto quadrature rules (up to fifth degree) were then explored using collocation point selection based on a family of Jacobi polynomials, resulting in an improvement in computational efficiency and solution accuracy \cite{Herman1996}. Building on these early methods, fourth-order compressed and separated forms of the Hermite-Simpson method were developed to further exploit sparsity properties in the problem differential equations and reduce the problem dimension \cite{Betts1999}. The family of Hermite-Lobatto collocation methods have been shown to achieve relatively high accuracy solutions for problems in which the system dynamics are well-behaved and stability for long-duration integration is required \cite{Williams2009,Liu2014,Pezent2024}. 

In recent years, the class of {\em Gaussian quadrature orthogonal collocation methods} have become increasingly popular for solving optimal control problems. Well-developed Gaussian quadrature orthogonal collocation methods employ Legendre-Gauss (LG) points\cite{Benson2006,Garg2010,Garg2011a}, Legendre-Gauss-Radau (LGR) points\cite{Garg2011a,Garg2011b}, or Legendre-Gauss-Lobatto (LGL) points\cite{Elnagar1995,Fahroo2001,Fahroo2008,Gong2008a,Garrido2023}.  While these methods may appear to be similar, the LG and LGR methods described in Refs.~\cite{Benson2006,Garg2010,Garg2011a,Garg2011b} along with the LGL method of Ref.~\cite{Garrido2023} differ fundamentally from the LGL method developed in Refs.~\cite{Elnagar1995,Fahroo2001,Fahroo2008,Gong2008a}.  In particular, the LG and LGR methods given in Refs.~\cite{Benson2006,Garg2010,Garg2011a,Garg2011b} along with the LGL method of Ref.~\cite{Garrido2023} employ polynomial state approximations of degree $N$ using $N$ quadrature points (LG, LGR, or LGL points) together with an additional support point.  The approximations used in these three methods lead to non-square full-rank differentiation matrices, and it was shown for the LG and LGR methods of Refs.~\cite{Benson2006,Garg2010,Garg2011a,Garg2011b} that the derivative form has an equivalent integral form.  On the other hand, the LGL method of Refs.~\cite{Elnagar1995,Fahroo2001,Fahroo2008,Gong2008a} employs a polynomial state approximation of degree $N-1$ using just the $N$ LGL points.  This state approximation leads to a square rank-deficient differentiation matrix and the method has no equivalency between the differential and integral forms.  Furthermore, the transformed adjoint systems of the LG and LGR methods of Refs.~\cite{Benson2006,Garg2010,Garg2011a,Garg2011b} along with the LGL method given in Ref.~\cite{Garrido2023} are full-rank while the transformed adjoint system of the LGL method of Refs.~\cite{Elnagar1995,Fahroo2001,Fahroo2008,Gong2008a} is rank-deficient.  As a result, under certain assumptions of smoothness and coercivity, the LG and LGR methods have been proven to be convergent \cite{Hager2016,Hager2018,Hager2019b}. 

This paper describes a new method for solving optimal control problems using collocation at Legendre-Gauss-Lobatto points.  The approach developed in this paper differs significantly from the approaches used for both the LG and LGR collocation methods and the approaches used for either of the previously developed LGL methods.  While these previously developed methods were based on differentiating a polynomial state approximation, the approach of this paper employs a polynomial approximation of the right-hand side vector field of the differential equations.  The integrals of this polynomial approximation of the right-hand side vector field from the initial point of the interval to each LGL point are then replaced with a Legendre-Gauss-Lobatto quadrature which results in a square integration matrix whose first row is zero \cite{Axelsson1964} because the first LGL point lies at the start of the interval and, thus, the integral to the first LGL point is zero.  Removing the first row of this square integration matrix \cite{Axelsson1964}, the remaining nonsquare portion of the integration matrix is used to cast the method into a {\em derivative-like} form.  This derivative-like form is obtained by multiplying the integral form by the inverse of the matrix formed from deleting the first column of the nonsquare portion of the integration matrix \cite{Axelsson1964}. The first-order optimality conditions for both the integral form and the derivative-like form are developed, and each of these forms leads to a costate estimate.  It is found via examples and analysis that the state, control, and costate converge exponentially for single-interval problems while for a multiple-interval form, the state and control converge at the superconvergence rate of $\mathcal{O}(h^{2N-2})$.  Finally, using the solution from the method of this paper, a second costate can be computed at the mesh points that converges at the same superconvergence rate as the state and control.  

Using the integral form of the LGL method of this paper,  a second integral form of LGL collocation is derived.  This second integral form of LGL collocation adds to the method developed in this paper an integral from the initial point of a mesh interval to a point distinct from any of the LGL points, and this additional integral is replaced with an LGL quadrature approximation.  Augmenting this additional quadrature approximation to the integral form developed in this paper results in a the second integral form that employs a square and full-rank integration matrix.  Using the fact that this second integration matrix is invertible, it is shown in this paper that this second integral form is equivalent to the derivative form found in Ref.~\cite{Garrido2023}.  Finally, this second integral form of LGL collocation is used to show that the inclusion of the noncollocated point as employed in Ref.~\cite{Garrido2023} is superfluous.  

The contributions of this research are as follows.  First, a new Legendre-Gauss-Lobatto collocation method for the numerical solution of optimal control problems is developed via an appropriate degree Lagrange polynomial approximation of the differential equation vector field.  Then, using this polynomial approximation of the vector field, an integral form of LGL collocation is developed where the dynamics are integrated to each LGL point.  These integrals are then replaced exactly with LGL quadrature in terms of the nonsquare full-rank integration matrix of Ref.~\cite{Axelsson1964}.  The first-order optimality conditions of the resulting nonlinear programming problem are then derived. These first-order optimality conditions then lead to a novel costate estimate.  Furthermore, a related costate estimate is obtained using a derivative-like form of the collocation equations, where this derivative-like form is obtained by multiplying the system by the inverse of an appropriate full-rank block of the integration matrix.  Moreover, the costate estimate obtained in either integral or derivative-like form are then related to one another.  Next, a second integral form of LGL collocation is developed using the integral form of LGL collocation developed in this paper, and this second integral form is shown to be equivalent to the derivative form developed in Ref.~\cite{Garrido2023} where Ref.~\cite{Garrido2023} employs an additional noncollocated point. Furthermore, the second integral form developed in this paper is used to show that the noncollocated point employed in Ref.~\cite{Garrido2023} is superfluous.  

The paper is organized as follows. Section~\ref{sect:Notation} provides the notation and conventions used throughout this paper.  Section~\ref{sect:OCP} describes the unconstrained continuous optimal control problem. Section~\ref{sect:LGL int} derives the integral form of the LGL collocation method via polynomial approximation of the right-hand side vector field of the differential equations of motion. Section~\ref{sect:LGL deriv} reformulates the method of Section~\ref{sect:LGL int} into a derivative-like form. Section~\ref{sect:LGL Connections} discusses the connections between the two forms, and Section~\ref{sect: arxiv comparison and new results} augments the integral form of this paper in order to derive a second integral form.  Section~\ref{sect:examples} demonstrates the method of this paper on two benchmark problems from the open literature and provides an analysis of the solution accuracy. Finally, Section~\ref{sect:conclusion}, provides conclusions on this research.  

\section{Notation and Conventions}\label{sect:Notation}

In this paper, the following notation and conventions will be used. 
First, the following conventions are used to specify certain elements of matrix $\m{A}$:
\begin{equation}\nonumber
    \begin{aligned}
        {A}_{ij} :={} & \tx{element in row \emph{i} and column \emph{j}}, \\\nonumber
        \m{A}_{(:,i)} :={} & \tx{elements in all rows and column \emph{i}}, \\\nonumber
        \m{A}_{(i,:)} :={} & \tx{elements in all columns and row \emph{i}}, \\\nonumber
        \m{A}_{(i:j,k:l)} :={} & \tx{elements in rows \emph{i} through \emph{j} and columns \emph{k} through \emph{l}}.
    \end{aligned}
\end{equation}
If $\m{z}_i\in\R^n$ is an $n$-length vector and $\m{Z}$ is a matrix whose $i^{\tx{th}}$ row is given by $\m{z}_i$, then the notation $\m{Z}_i$ is used to denote \emph{row} $i$ of $\m{Z}$, while $\m{Z}_{i:j}$ is used to denote \emph{rows} $i$ through $j$ of $\m{Z}$ for $j>i$.
If $\m{W}$ is a diagonal matrix with $i^{\tx{th}}$ diagonal element $w_i,~(i=1,\ldots,n)$, then $\m{W}_{j:k}$ denotes the submatrix $\m{W}_{(j:k,j:k)}$ for $k>j$.  Given two matrices $\m{A}$ and $\m{B}$ of compatible dimensions, $\langle \m{A},\m{B}\rangle$ is the Frobenius inner product:
\[
\langle \m{A},\m{B}\rangle = \tx{tr}(\m{A}\T\m{B}),
\]
where the notation $\m{A}\T$ denotes the transpose of a matrix $\m{A}$.
When $\m{A}$ and $\m{B}$ are vectors, this is the standard vector inner product. The boldface symbol $\m{1}$ is used to denote a column vector of all ones, the boldface symbol $\m{0}$ is used to denote either a vector or matrix of all zeros, and $\m{I}_q$ denotes a $q\times q$ identity matrix.

Next, $z(\tau)\in\R$ denotes a scalar function $z$ of the independent variable $\tau$. Then, $\m{z}(\tau)\in\R^n$ denotes a vector function of $\tau$ with dimension $n$, such that  $\m{z}(\tau_i) := [z_1(\tau_i),z_2(\tau_i),\ldots,z_{n-1}(\tau_i),z_n(\tau_i)]$ is a \emph{row} vector and $\tau_i$ is a discrete time step of $\tau$. Furthermore, the derivative of a vector function $\m{z}(\tau)$ with respect to $\tau$, denoted by $\Dot{\m{z}}(\tau)$, is given as $\tx{d}\m{z}(\tau)/\tx{d}\tau := \Dot{\m{z}}(\tau) :=[\Dot{z}_1(\tau),\Dot{z}_2(\tau),\ldots,\Dot{z}_{n-1}(\tau),\Dot{z}_n(\tau)]$. A vector function $\m{z}(\tau)\in\R^n,~n\geq 1$, evaluated at a set of discrete points $\{\tau_i ~|~ i=1,\ldots,N\}$ is defined by
\[ \m{z}(\tau_{1:N}):= \begin{bmatrix}
    z_1(\tau_1) & z_2(\tau_1) & \cdots & z_n(\tau_1) \\
    z_1(\tau_2) & z_2(\tau_2) & \cdots & z_n(\tau_2)\\
    \vdots & \vdots & \ddots & \vdots\\
    z_1(\tau_N) & z_2(\tau_N) & \cdots & z_n(\tau_N)
\end{bmatrix} \in\R^{N\times n}. \]
Now, if $\m{f}:\R^n \rightarrow \R^m$ is a vector function of $\m{x}\in\R^n$, then $\nabla_x (\m{f})$ is the $m\times n$ Jacobian matrix whose $i^{\tx{th}}$ row is $\nabla_{x} (f_i)$. That is, the gradient of a scalar-valued function $f:\R^n\rightarrow\R$ is a row vector given by 
\[
\nabla_{x} (f)  = \left[ \deldt{f}{x_1}, \deldt{f}{x_2}, \ldots, \deldt{f}{x_n} \right] \in \R^n,
\]
and the Jacobian of a vector-valued function $\m{f}:\R^n\rightarrow \R^m$ is 
\[
\nabla_{x} (\m{f})  = \left[ \deldt{f_1}{\m{x}}, \deldt{f_2}{\m{x}}, \ldots, \deldt{f_m}{\m{x}} \right]\T \in \R^{m\times n}.
\]
Finally, if $f:\R^{m\times n}\rightarrow\R$ and input $\m{X}$ is an $m\times n$ matrix, then $\nabla_X (f)$ is the $m\times n$ matrix whose $(i,j)$ element is given by $(\nabla_{X} (f))_{ij} = \partial f(\m{X}_i)/\partial X_{ij}$.

\section{Continuous Optimal Control Problem}\label{sect:OCP}
To simplify the derivation, consider the following unconstrained optimal control problem on the domain $\tau\in[-1,+1]$. 
\begin{equation}\label{OCP}
    \tx{minimize } \Phi(\m{x}(+1)) \tx{ subject to } \left\{ \begin{aligned}
        \Dot{\m{x}}(\tau) &= \m{f}(\m{x}(\tau),\m{u}(\tau)), \\
        \m{x}(-1) &= \m{x}_0,
    \end{aligned} \right.
\end{equation}
where $\m{x}(\tau)\in\R^{n_x}$ is the state, $\m{u}(\tau) \in \R^{n_u}$ is the control, the functions $\Phi$ and $\m{f}$ are defined by the mappings $\Phi:\R^{n_x}\rightarrow \R$ and $\m{f}:\R^{n_x} \times \R^{n_u} \rightarrow \R^{n_x}$, and $\m{x}_0$ is the initial condition, which is assumed to be given. It is noted that the computational domain $\tau\in[-1,+1]$ can be transformed to the time interval $t\in[t_0,t_f]$ via the affine transformation
\begin{equation}\label{affine trans}
    t = \frac{t_f-t_0}{2}\tau + \frac{t_f+t_0}{2}.
\end{equation}

Following Pontryagin's minimum principle \cite{Pontryagin1962}, the first-order optimality conditions of the continuous-time optimal control problem in Eq.~\eqref{OCP} are given by
\begin{align}\label{1st order opt x}
    \dot{\m{x}}(\tau) &= \m{f}(\m{x}(\tau),\m{u}(\tau)), \\ \label{1st order opt lambda}
    \dot{\g{\lambda}}(\tau) & = -\nabla_x \left\langle \g{\lambda}(\tau),\m{f}(\m{x}(\tau),\m{u}(\tau)) \right\rangle, \\ \label{1st order opt trans}
    \g{\lambda}(+1) &= \nabla \Phi(\m{x}(+1)), \\ \label{1st order opt u}
    \m{0} &= \nabla_u \left\langle \g{\lambda}(\tau),\m{f}(\m{x}(\tau),\m{u}(\tau)) \right\rangle,
\end{align}
where $\g{\lambda}\in\R^{n_x}$ is the costate.

\section{Integral Form of Legendre-Gauss-Lobatto Collocation}\label{sect:LGL int}
First, the right-hand side vector field $\m{f}(\m{x}(\tau),\m{u}(\tau))$ in Eq.~\eqref{1st order opt x} is approximated as a polynomial of degree at most $N-1$ using a basis of Lagrange polynomials supported at the $N$ LGL points $(\tau_1,\ldots,\tau_N)\in[-1,+1]$.  First, let $\m{X}_i\approx \m{x}(\tau_i)\in\R^{n_x}$ and $\m{U}_i\approx \m{u}(\tau_i)\in\R^{n_u},~(i=1,\ldots,N),$ be the approximations of the state and the control, respectively, at each of the $N$ LGL points, $(\tau_1,\ldots,\tau_N)$. Next, let $\hat{\m{f}}$ be an approximation of $\m{f}(\m{x}(\tau),\m{u}(\tau))$ from Eq.~\eqref{1st order opt x} such that
\begin{equation}\label{f approx}
    \m{f}(\m{x}(\tau),\m{u}(\tau)) \approx \hat{\m{f}}(\tau)= \sum_{j=1}^N \m{F}_jL_j(\tau), \quad \m{F}_j = \m{f}(\m{X}_j,\m{U}_j),
\end{equation}
where
\begin{equation}\label{Lagrange basis for f}
    L_j(\tau) = \prod_{\substack{i = 1 \\ i \neq j}}^N \frac{\tau - \tau_i}{\tau_j - \tau_i}, \quad (j=1,\ldots,N),
\end{equation}
are the Lagrange polynomials of degree $N-1$ with $N\geq 2$. Furthermore, the Lagrange polynomials satisfy the isolation property
\begin{equation}
    L_j(\tau_i) = \delta_{ij} = \begin{cases}
        1, & i=j, \\ 0, & i\neq j.
    \end{cases}
\end{equation}

Now, because $\hat{\m{f}}$ is a polynomial of degree $N-1$ with $N\geq 2$, its integral from $-1$ to $\tau_i$ can be evaluated exactly by LGL quadrature such that 
\begin{equation}\label{int approx f}
    \int_{-1}^{\tau_i} \hat{\m{f}}(\tau)\dt{\tau} = \sum_{j=1}^N A_{ij} \m{F}_j, \quad (i=1,\ldots,N),
\end{equation}
where 
\begin{equation}\label{Axelsson integration matrix elements}
    A_{ij} := \int_{-1}^{\tau_i} L_j(\tau)\dt{\tau}, \quad (i,j=1,\ldots,N),
\end{equation}
are the $(i,j)$-elements of the $N\times N$ integration matrix $\m{A}$ originally defined by Axelsson\cite{Axelsson1964}. More precisely, $\m{A}$ is an integration matrix for the space of polynomials of degree at most $N-1$.

Next, if the state derivative is approximated as $\dot{\m{x}}(\tau) \approx \dot{\m{X}}(\tau)$ and $\m{X}_i$ is defined by $\m{X}_i:=\m{X}(\tau_i)$, then
\begin{equation}
\begin{aligned}\label{int approx xdot}
    \int_{-1}^{\tau_i} \dot{\m{X}}(\tau)\dt{\tau} &= \m{X}(\tau_i) - \m{X}(-1), \\
    &= \m{X}_i - \m{X}_1.
\end{aligned}
\end{equation}
Thus, using the results from Eq.~\eqref{int approx f} and Eq.~\eqref{int approx xdot}, the state at $\tau_i$ is approximated by
\begin{equation}\label{xdot integrated}
    \m{X}_i = \m{X}_1 + \sum_{j=1}^N A_{ij} \m{F}_j, \quad (i=1,\ldots,N).
\end{equation}
The first row of $\m{A}$ is zero because the integration from $-1$ to $\tau_1=-1$ is itself zero. Therefore, Eq.~\eqref{xdot integrated} can be compactly written with this triviality removed as
\begin{equation}
    \begin{bmatrix}
        \m{X}_2 \\ \vdots \\ \m{X}_N
    \end{bmatrix} = \m{1X}_1 + \tilde{\m{A}} \begin{bmatrix}
        \m{F}_1 \\ \vdots \\ \m{F}_N
    \end{bmatrix},
\end{equation}
which can be written equivalently as
\begin{equation}\label{LGL integral form}
    \m{X}_{2:N} = \m{1}\m{X}_{1} + \tilde{\m{A}}\m{F},
\end{equation}
where $\tilde{\m{A}}:= \m{A}_{(2:N,:)}\in\R^{(N-1)\times N}$ is the submatrix that is formed by omitting the first row of $\m{A}$. 

The finite-dimensional nonlinear programming problem (NLP) associated with the integral form of the presented LGL method is then given as 
\begin{equation}\label{integral NLP}
    \tx{minimize } \Phi(\m{X}_{N}) \tx{ subject to } \left\{ \begin{aligned}
        \m{0} &= \begin{bmatrix} \m{1} & -\m{I}_{N-1} \end{bmatrix} \m{X} + \tilde{\m{A}}\m{F}, \\ \m{X}_{1} &= \m{x}_0,
    \end{aligned}\right.
\end{equation}
where $\m{F}=\m{f}(\m{X},\m{U})\in\R^{N\times n_x}$ denotes the right-hand side of the state dynamics in Eq.~\eqref{1st order opt x} evaluated at the LGL points, and $\m{X}\in\R^{N\times n_x}$ and $\m{U}\in\R^{N\times n_u}$ denote matrices of the state and control approximations, respectively, at the LGL points.  That is, $\m{X}_i\in\R^{n_x}$ is the $i^{\tx{th}}$ row of $\m{X}$, $\m{U}_i\in\R^{n_u}$ is the $i^{\tx{th}}$ row of $\m{U}$, and $\m{F}_i:=\m{f}(\m{X}_i,\m{U}_i)\in\R^{n_x}$ is the $i^{\tx{th}}$ row of $\m{F}$.

\subsection{KKT System of the LGL Integral Form}

The Lagrangian of the NLP given by Eq.~\eqref{integral NLP} is then
\begin{equation}\label{integral Lagrangian}
\C{L} = \Phi(\m{X}_N) + \left\langle \m{M}, \begin{bmatrix} \m{1} & -\m{I}_{N-1} \end{bmatrix} \m{X} + \tilde{\m{A}}\m{F} \right\rangle + \left\langle \g{\mu}, \m{x}_0 - \m{X}_1\right\rangle,
\end{equation}
where $\m{M}\in\R^{(N-1)\times n_x}$ and $\g{\mu}\in\R^{n_x}$ are NLP Lagrange multipliers corresponding to the dynamics constraints and initial conditions, respectively. 

The Karush-Kuhn-Tucker (KKT) optimality conditions for the NLP of Eq.~\eqref{integral NLP} are then found by setting the partial derivatives of the Lagrangian in Eq.~\eqref{integral Lagrangian} equal to zero. The KKT conditions are given as
\begin{align}\label{integral KKT split x1}
    \g{\mu} &= \m{1}\T\m{M} + \nabla_{X_1}\left\langle \m{M}, \tilde{\m{A}}_{(:,1)}\m{F}_1 \right\rangle, \\ \label{integral KKT split x2thruN}
    \m{M}_{} &= \nabla_{X_{2:N}} \left\langle \m{M}, \tilde{\m{A}}_{(:,2:N)}\m{F}_{2:N}  \right\rangle + \tilde{\m{e}}_{N-1}\nabla_{X_N}\Phi(\m{X}_N), \\ \label{integral KKT split u}
    \m{0}&= \nabla_U \left\langle \m{M}, \tilde{\m{A}}\m{F} \right\rangle,
\end{align}
where $\tilde{\m{e}}_{N-1}$ is the last column of the $\m{I}_{N-1}$ identity matrix. 

\subsection{Transformed KKT System of the LGL Integral Form}\label{subsect: transformed KKT int}
Let $\m{W}$ denote the $N\times N$ diagonal matrix with $i^{\tx{th}}$ diagonal element $w_i$, where $w_i,~(i=1,\ldots,N),$ is the $i^{\tx{th}}$ LGL quadrature weight. Next, let $\g{\Lambda}\in\R^{N\times n_x}$ denote the costate approximation at the LGL points, defined as
\begin{equation}\label{Lambda int}
    \g{\Lambda} = \m{W}\inv \tilde{\m{A}}\T\m{M}.
\end{equation}
Although the dimension of the NLP Lagrange multipliers corresponding to the state dynamics integration is $(N-1)\times n_x$, the dimension of the costate approximation is $N\times n_x$. That is, the costate is approximated at all $N$ LGL points.

Now, in order to connect the discrete costate equations given by Eqs.~\eqref{integral KKT split x1}-\eqref{integral KKT split x2thruN} to the continuous costate equations given by Eq.~\eqref{1st order opt lambda}, an $N\times N$ matrix $\m{A}\sdag$, which is a modified version of $\m{A}$, is defined as follows:
\begin{equation}\label{Adag}
    A\sdag_{ij} = w_j - \frac{w_j}{w_i}(1-{\delta}_{Ni})A_{ji} , \quad  (i,j=1,\ldots,N).
\end{equation}
It is noted that $A\sdag_{ij},~(i=1,\ldots,N-1;j=1,\ldots,N),$ follows from the integration matrix of the transformed adjoint system for symplectic Runge-Kutta schemes derived by Hager \cite{Hager2025}. Additionally, $A\sdag_{Nj}\equiv A_{Nj} = w_j,~(j=1,\ldots,N)$, are equal to the LGL quadrature weights. The transformed adjoint system corresponding to the integral form of the LGL collocation scheme developed in this work is then
\begin{equation}\label{integral transformed adj}
    \g{\Lambda} = \m{1}\g{\Lambda}_1 + {\m{A}}\sdag \left[ -\nabla_{X}\left\langle \g{\Lambda}, \m{F}\right\rangle + \frac{\m{e}_1}{w_1}(\g{\mu} - \g{\Lambda}_1)+\frac{\m{e}_N}{w_N}(\g{\Lambda}_N - \nabla \Phi(\m{X}_N)) \right],
\end{equation}        
where $\m{e}_1$ and $\m{e}_N$ are the first and last columns, respectively, of $\m{I}_{N}$, and $\m{A}\sdag\in\R^{N\times N}$ is shown to be an integration matrix for the space of polynomials of degree at most $N-3$ in Section~\ref{subsect: transformed adjoint relation}. Substituting Eqs.~\eqref{Lambda int}-\eqref{Adag} into Eq.~\eqref{integral transformed adj}, the KKT conditions of Eqs.~\eqref{integral KKT split x1}-\eqref{integral KKT split x2thruN} are obtained. 

Equation~\eqref{integral transformed adj} approximates the costate $\g{\lambda}(\tau_i)\approx \g{\Lambda}_i,~(i=1,\ldots,N)$, by integrating the discrete costate dynamics from -1 to $\tau_i$, assuming the discrete transversality condition $\g{\Lambda}(+1)=\nabla \Phi(\m{X}_N)$ is satisfied. If $\g{\Lambda}(+1)=\nabla \Phi(\m{X}_N)$ is satisfied, then the continuous transversality condition of Eq.~\eqref{1st order opt trans} is satisfied. Furthermore, satisfying the condition $\g{\mu}=\g{\Lambda}_1$ implies that the continuous costate $\g{\lambda}(-1)$ is free, exactly as in the continuous optimality conditions given in Eqs.~\eqref{1st order opt x}-\eqref{1st order opt u}. In practice, if the continuous costate solution is smooth but is not a polynomial, $\g{\Lambda}(+1) \approx \nabla \Phi(\m{X}_N)$ and $\g{\mu}\approx\g{\Lambda}_1$ as $N$ is increased. As the expressions $(\g{\mu}-\g{\Lambda}_1)$ and $(\g{\Lambda}_N - \nabla \Phi(\m{X}_N))$ both approach zero, then $-\nabla_X \left\langle \g{\Lambda},\m{F} \right\rangle = \dot{\g{\Lambda}}$ becomes an increasingly better approximation of the continuous costate dynamics $\dot{\g{\lambda}}$ in Eq.~\eqref{1st order opt lambda}. Lastly, the transformed KKT condition
\begin{equation}
    \m{0} = \nabla_U \left\langle \g{\Lambda}, \m{F}\right\rangle,
\end{equation}
maps to the continuous optimal control law given by Eq.~\eqref{1st order opt u}.

\section{Derivative-Like Form of Legendre-Gauss-Lobatto Collocation}\label{sect:LGL deriv}
The integral form of the LGL collocation method in Eq.~\eqref{xdot integrated} cannot be written in an equivalent differentiation form because $\m{A}\in\R^{N\times N}$ is singular (recall, the first row of $\m{A}$ is zero and $\tx{rank}(\m{A})=N-1$). However, the submatrix $\tilde{\m{A}}_{(:,2:N)}\in\R^{(N-1)\times(N-1)}$ is rank $N-1$ and can therefore be inverted to transform Eq.~\eqref{LGL integral form} into a {\em derivative-like form}, where this derivative-like form closely resembles a derivative form. Left-multiplying both sides of Eq.~\eqref{LGL integral form} by $[\tilde{\m{A}}_{(:,2:N)}]\inv$ and rearranging, the LGL collocation method of this paper can be written equivalently as
\begin{equation}\label{pseudoderiv defect}
    \m{E}\begin{bmatrix}
        \m{X}_1 \\ \vdots \\ \m{X}_N
    \end{bmatrix} = \begin{bmatrix}
        [\tilde{\m{A}}_{(:,2:N)}]\inv \tilde{\m{A}}_{(:,1)} & \m{I}_{N-1}
    \end{bmatrix} \begin{bmatrix}
        \m{F}_1 \\ \vdots \\ \m{F}_N
    \end{bmatrix}, \quad \m{E}:= [\tilde{\m{A}}_{(:,2:N)}]\inv\begin{bmatrix}
        -\m{1} & \m{I}_{N-1}
    \end{bmatrix},
\end{equation}
where $\m{E}\in\R^{(N-1)\times N}$ is generally {\emph{not}} a differentiation matrix because $\m{Ey}\neq \dot{\m{y}}$ for any polynomial $\m{y}$ unless the derivative vanishes at $\tau_1$. The matrix $\m{E}$ has the following properties: (a) $\m{E1}=\m{0}$ and (b) $\m{E}_{(:,2:N)}\m{A}_{(2:N,2:N)}=\m{I}_{N-1}$.

\begin{lemma}\label{lemma E1=0}
    $\m{E1}=\m{0}$; equivalently, $-\m{E}_{(:,2:N)}\m{1}=\m{E}_{(:,1)}$.
\end{lemma}
\begin{proof}
    $\m{E}$, defined in Eq.~\eqref{pseudoderiv defect}, is given by $\m{E}:=\begin{bmatrix}
        -[\tilde{\m{A}}_{(:,2:N)}]\inv\m{1} & [\tilde{\m{A}}_{(:,2:N)}]\inv
    \end{bmatrix}$. Then, 
    \begin{equation}
        \begin{aligned}
            \m{E1} &= \begin{bmatrix}
        -[\tilde{\m{A}}_{(:,2:N)}]\inv\m{1} & [\tilde{\m{A}}_{(:,2:N)}]\inv
    \end{bmatrix}\m{1} \\
            &= -[\tilde{\m{A}}_{(:,2:N)}]\inv\m{1} +[\tilde{\m{A}}_{(:,2:N)}]\inv\m{1} \\
            &= \m{0}.
        \end{aligned}
    \end{equation}
    Furthermore, $\m{E1}\equiv\m{E}_{(:,1)}+\m{E}_{(:,2:N)}\m{1}$ which implies $-\m{E}_{(:,2:N)}\m{1}=\m{E}_{(:,1)}$.
\end{proof}

\begin{lemma}\label{lemma EA=I}
    $\m{E}_{(:,2:N)}\m{A}_{(2:N,2:N)}=\m{I}_{N-1}$; equivalently, $\m{E}_{(i,2:N)}\m{A}_{(2:N,2:N)}=\tilde{\m{e}}_i\T, ~(i=1,\ldots,N-1)$.
\end{lemma}
\begin{proof}
    From the definition of $\m{E}$ in Eq.~\eqref{pseudoderiv defect}, $\m{E}_{(:,2:N)}= [\m{A}_{(2:N,2:N)}]\inv$. Thus,
    \begin{equation}
        \begin{aligned}
            \m{E}_{(:,2:N)}\m{A}_{(2:N,2:N)} &= [\m{A}_{(2:N,2:N)}]\inv \m{A}_{(2:N,2:N)} \\ &= \m{I}_{N-1}.
        \end{aligned}
    \end{equation}
    If $\m{E}_{(:,2:N)}$ is replaced with just the $i^{\tx{th}}$ row of $\m{E}$, then $\m{E}_{(i,2:N)}\m{A}_{(2:N,2:N)}=\tilde{\m{e}}_i\T$, where $\tilde{\m{e}}_i\T$ is the $i^{\tx{th}}$ row of the $(N-1)\times(N-1)$ identity matrix, $\m{I}_{N-1}$.
\end{proof}

The finite-dimensional NLP associated with the derivative-like form of the LGL method developed in this paper is then given as
\begin{equation}\label{pseudoderivative NLP}
    \tx{minimize } \Phi(\m{X}_N) \tx{ subject to } \left\{ \begin{aligned}
        \m{0} &= \begin{bmatrix}
        \g{\alpha} & \m{I}_{N-1}
    \end{bmatrix}\m{F} - \m{EX}, \\ \m{X}_1 &= \m{x}_0,
    \end{aligned}\right.
\end{equation}
where $\g{\alpha}\in\R^{(N-1)}$ is defined as
\begin{equation}\label{alpha}
    \g{\alpha} := [\tilde{\m{A}}_{(:,2:N)}]\inv \tilde{\m{A}}_{(:,1)}.
\end{equation}
Although the derivative-like form of Eq.~\eqref{pseudoderivative NLP} is generally not a true derivative form, the associated NLP given by Eq.~\eqref{pseudoderivative NLP} results in a sparser constraint Jacobian and Lagragian Hessian compared with the NLP obtained using the integral form given in Eq.~\eqref{integral NLP}. 

\subsection{KKT System of the LGL Derivative-Like Form}
The Lagrangian of the NLP given by Eq.~\eqref{pseudoderivative NLP} is then
\begin{equation}\label{pseudoderiv Lagrangian}
\C{L} = \Phi(\m{X}_N) + \left\langle \m{S}, \begin{bmatrix}
        \g{\alpha} & \m{I}_{N-1}
    \end{bmatrix}\m{F} - \m{EX} \right\rangle + \left\langle \g{\mu}, \m{x}_0 - \m{X}_1\right\rangle,
\end{equation}
where $\m{S}\in\R^{(N-1)\times n_x}$ and $\g{\mu}\in\R^{n_x}$ are NLP Lagrange multipliers corresponding to the dynamics constraints and initial conditions, respectively. 

The KKT optimality conditions for the NLP of Eq.~\eqref{pseudoderivative NLP} are then found by setting the partial derivatives of the Lagrangian in Eq.~\eqref{pseudoderiv Lagrangian} equal to zero. The KKT conditions are given as
\begin{align}\label{pseudoderiv KKT split x1}
    \m{E}\T_{(:,1)}\m{S} &= \nabla_{X_1}\left\langle \m{S}, \g{\alpha}\m{F}_1 \right\rangle - \g{\mu}, \\ \label{pseudoderiv KKT split x2thruN}
    \m{E}\T_{(:,2:N)}\m{S} &= \nabla_{X_{2:N}} \left\langle \m{S},\m{F}_{2:N}  \right\rangle + \tilde{\m{e}}_{N-1}\nabla_{X_N}\Phi(\m{X}_N), \\ \label{pseudoderiv KKT split u}
    \m{0}&= \nabla_U \left\langle \m{S}, \begin{bmatrix}
        \g{\alpha} & \m{I}_{N-1}
    \end{bmatrix}\m{F} \right\rangle.
\end{align}

\subsection{Transformed KKT System of the LGL Derivative-Like Form}
Let $\g{\Lambda}\in\R^{N\times n_x}$ denote the costate approximation at the LGL points, defined as
\begin{equation}\label{Lambda pseudoderiv}
\begin{aligned}
    \g{\Lambda}_1 &= \frac{1}{w_1} \g{\alpha}\T\m{S},\\
    \g{\Lambda}_{2:N} &= \m{W}_{2:N}\inv \m{S},
\end{aligned}
\end{equation}
where $\m{W}_{2:N}$ is an $(N-1)\times(N-1)$ diagonal matrix with $j^{\tx{th}}$ diagonal element equal to $w_{j+1},~(j=1,\ldots,N-1)$ and $\g{\alpha}$ is defined by Eq.~\eqref{alpha}.  Once again, the dimension of the NLP Lagrange multipliers corresponding to the collocation constraints is $(N-1)\times n_x$ but the dimension of the costate approximation is $N\times n_x$. That is, the costate is approximated at all $N$ LGL points such that $\g{\lambda}(\tau_i)\approx \g{\Lambda}_i,~(i=1,\ldots,N)$.

Next, an explicit relationship for the components of $\g{\alpha}$ will be developed and it will be convenient for this derivation to denote the $N-1$ components of $\g{\alpha}$ by $\alpha_j,\; (j=2,\ldots,N)$.
\begin{lemma}\label{lemma-alpha}
The vector $\g{\alpha}$ defined in $(\ref{alpha})$ is given by
\begin{equation}\label{solution}
\alpha_j = -\frac{\dot{\C{L}}_N(\tau_j)}{\dot{\C{L}}_N(\tau_1)}, \quad (j=2,\ldots,N),
\end{equation}
where $\C{L}_N(\tau)$ is defined as
\begin{equation}\label{Lobatto polynomial}
    \C{L}_N(\tau) = (\tau^2-1)\dot{P}_{N-1}(\tau),
\end{equation}
and $P_{N-1}(\tau)$ is the $(N-1)^{\tx{th}}$-degree Legendre polynomial.
\end{lemma}
\begin{proof}
In terms of the Lagrange polynomials $L_j$ defined in Eq.~\eqref{Lagrange basis for f} and the integration coefficients $A_{ij}$ defined in Eq.~\eqref{Axelsson integration matrix elements}, the $i^{\tx{th}}$ equation, $(i=2,\ldots,N)$, in the system
\[
{\m{A}}_{(:,2:N)} \g{\alpha} - {\m{A}}_{(:,1)} = \m{0}
\]
has the form
\begin{equation}\label{lhs=rhs-alpha}
\int_{-1}^{\tau_i} ~
\sum_{j=2}^N L_j(\tau) \alpha_j \dt{\tau} -
\int_{-1}^{\tau_i} L_1 (\tau) \dt{\tau} = 0.
\end{equation}
Substituting the proposed solution into Eq.~\eqref{lhs=rhs-alpha} gives
\begin{equation}\label{lhs=rhs-alpha-2}
-\int_{-1}^{\tau_i} \sum_{j=2}^N L_j(\tau) \frac{\dot{\C{L}}_N(\tau_j)}{\dot{\C{L}}_N(\tau_1)}
\dt{\tau} - \int_{-1}^{\tau_i} L_1 (\tau) \dt{\tau} = -\int_{-1}^{\tau_i} \sum_{j=2}^N L_j(\tau) \frac{\dot{\C{L}}_N(\tau_j)}{\dot{\C{L}}_N(\tau_1)}
\dt{\tau} - \frac{\dot{\C{L}}_N(\tau_1)}{\dot{\C{L}}_N(\tau_1)}\int_{-1}^{\tau_i} L_1 (\tau) \dt{\tau},
\end{equation}
which simplifies to 
\begin{equation}\label{integral-lemma-alpha}
- \frac{1}{\dot{\C{L}}_N(\tau_1)} \int_{-1}^{\tau_i} ~
\sum_{j=1}^N L_j(\tau)  \dot{\C{L}}_N (\tau_j) \dt{\tau}.
\end{equation}
Now, because $\dot{\C{L}}_N(\tau)$ is a polynomial of degree $N-1$, it is seen that
\begin{equation}\label{Lobatto-poynomial-derivative}
    \dot{\C{L}}_N(\tau) = \sum_{j=1}^N L_j(\tau)  \dot{\C{L}}_N (\tau_j) \dt{\tau},
\end{equation}
from which Eq.~\eqref{integral-lemma-alpha} can be written as
\begin{equation}\label{integral-lemma-alpha-2}
- \frac{1}{\dot{\C{L}}_N(\tau_1)} \int_{-1}^{\tau_i} ~
\dot{\C{L}}_N(\tau) \dt{\tau} = \frac{\C{L}_N(\tau_i)-\C{L}_N(-1)}{\dot{\C{L}}_N(\tau_1)}.
\end{equation}
Because both endpoints of the integral are LGL points, $\C{L}_N(\tau_i)=\C{L}_N(-1)=0$, completing the proof.  
\end{proof}

Now, the discrete costate equations given by Eqs.~\eqref{pseudoderiv KKT split x1} and \eqref{pseudoderiv KKT split x2thruN} are mapped to the continuous costate equations given by Eq.~\eqref{1st order opt lambda}. First, an $N\times (N-1)$ matrix $\m{D}\sdag$, which is a modified version of $\m{E}$, is defined as follows:
\begin{equation}\label{Ddag}
\begin{aligned}
    D\sdag_{1j} &= -\frac{w_{j+1}}{w_1}E_{j1} - \frac{w_{j+1}}{w_1^2}\alpha_j, &&(j=1,\ldots,N-1),\\
    D\sdag_{ij} &= -\frac{w_{j+1}}{w_i}E_{ji} + \frac{\delta_{Ni}\delta_{(N-1)j}}{w_N}, &&(i=2,\ldots,N;j=1,\ldots,N-1),
\end{aligned} 
\end{equation}
where now the $N-1$ components of $\g{\alpha}$ are denoted by $\alpha_j,\; (j=1,\ldots,N-1)$.
The transformed adjoint system corresponding to the derivative-like form of the LGL collocation scheme developed in this work is then
\begin{equation}\label{pseudoderiv transformed adj}
    \m{D}\sdag\g{\Lambda}_{2:N} = -\nabla_{X}\left\langle \g{\Lambda}, \m{F}\right\rangle + \frac{\m{e}_1}{w_1}(\g{\mu} - \g{\Lambda}_1)+\frac{\m{e}_N}{w_N}(\g{\Lambda}_N - \nabla \Phi(\m{X}_N)) ,
\end{equation}        
where $\m{D}\sdag\in\R^{N\times (N-1)}$ is a differentiation matrix for the space of polynomials of degree at most $N-2$ and as shown in Theorem~\ref{theorem Ddag} below. Equation~\eqref{pseudoderiv transformed adj} is obtained by combining Eqs.~\eqref{pseudoderiv KKT split x1} and \eqref{pseudoderiv KKT split x2thruN} with Eqs.~\eqref{Lambda pseudoderiv} and \eqref{Ddag}.
\begin{theorem}\label{theorem Ddag}
If $\m{v} \in \mathbb{R}^N$ is a vector with $v_1 =$ $w_1^{-1} \g{\alpha}\T\m{W}_{2:N}\m{v}_{2:N}$ and $v(\tau)$ is the polynomial of degree at most $N-2$ that satisfies $v(\tau_j) =v_j$, $(j = 1,\ldots,N)$, then
\begin{equation}\label{Ddag matrix}
  (\m{D}\sdag \m{v}_{2:N})_j = \dot{v}(\tau_j), \quad (j=1,\ldots,N).
\end{equation}
\end{theorem}
\begin{proof}   
After substituting for $\alpha_j$ using Lemma~\ref{lemma-alpha}, the expression for $v_1$ in the statement of the theorem can be rearranged into the identity
\begin{equation}
\sum_{j=1}^N w_j \dot{\C{L}}_N (\tau_j) v_j = 0.
\end{equation}
In Theorem 4 of Ref.~\cite{Garrido2023}, it is shown that when a vector $\m{v}$ satisfies this identity, the associated polynomial $v(\tau)$ that interpolates the components $v_j$ at the Lobatto points has degree at most $N-2$. Next, let $y$ be a polynomial of degree at most $N$ with derivative $\dot{y}$ which vanishes at $\tau_1$ and which satisfies $y(\tau_j) =$ $y_j,\: (j = 1, \ldots, N)$. Then, via integration by parts,
\begin{equation}\label{IbP}
    \int_{-1}^{+1} \dot{y}(\tau)v(\tau)\dt{\tau} = \Big[ y(\tau)v(\tau) \Big]_{-1}^{+1} - \int_{-1}^{+1} y(\tau)\dot{v}(\tau)\dt{\tau}.
\end{equation}
For $y$ defined as a polynomial of degree at most $N$ and $v$ defined as a polynomial of degree at most $N-2$ for $N\geq 3$, the polynomials given by $\dot{y}v$ and $y\dot{v}$ are of degree at most $2N-3$. Moreover, because Legendre-Gauss-Lobatto quadrature is exact for polynomials of degree at most $2N-3$, the integrals in Eq.~\eqref{IbP} can be replaced by their quadrature equivalents to obtain
    \begin{equation}\label{IbP quad}
        \sum_{j=1}^N w_j \dot{y}_j v_j = y_N v_N - y_1 v_1 - \sum_{j=1}^N w_j y_j \dot{v}_j.
    \end{equation}
    More compactly, Eq.~\eqref{IbP quad} can be expressed as
    \begin{equation}\label{IbP quad compact}
        (\m{W}\dot{\m{y}})\T \m{v} = y_N v_N - y_1 v_1 - (\m{Wy})\T \dot{\m{v}},
    \end{equation}
    where $\dot{\m{y}}\in\R^{N}$ is the vector with $j^{\tx{th}}$ component $\dot{y}_j=\dot{y}(\tau_j)$ and $\dot{\m{v}}\in\R^N$ is the vector with $j^{\tx{th}}$ component $\dot{v}_j=\dot{v}(\tau_j),~(j=1,\ldots,N)$ such that $\dot{\m{v}}=\m{Gv}_{2:N}$, where $\m{G}\in\R^{N\times(N-1)}$ is the differentiation matrix that takes into account that $v_1$ is a linear combination of $(v_2,\ldots,v_N)$.  
    By Eq.~\eqref{pseudoderiv defect}, with $\m{X}$ replaced by $\m{y}$ and $\m{F}$ replaced by $\dot{\m{y}}$, where $\dot{\m{y}}_1=\dot{\m{y}}(\tau_1) = 0$, the expression $\dot{\m{y}}_{2:N} = \m{Ey}$ is obtained.  Then, substituting $v_1$, $\dot{\m{v}}$, and $\dot{\m{y}}$ into Eq.~\eqref{IbP quad compact} yields
    \begin{equation}
        \m{y}\T \left( \m{E}\T \m{W}_{2:N} + \m{WG} + w_1\inv \m{e}_1 \g{\alpha}\T \m{W}_{2:N} - \m{e}_N\tilde{\m{e}}_{N-1}\T \right) \m{v}_{2:N} = \m{0},
    \end{equation}
    where $\m{e}_1$ and $\m{e}_N$ are the first and last columns, respectively, of $\m{I}_{N}$, and $\tilde{\m{e}}_{N-1}$ is the last column of $\m{I}_{N-1}$. Because $\m{y}$ and $\m{v}$ are arbitrary, it can be deduced that
    \begin{equation}
        \m{E}\T \m{W}_{2:N} + \m{WG} + w_1\inv \m{e}_1 \g{\alpha}\T \m{W}_{2:N} - \m{e}_N\tilde{\m{e}}_{N-1}\T = \m{0},
    \end{equation}
    which implies that 
    \begin{equation}\label{G}
    \begin{aligned}
        G_{1j} &= -\frac{w_{j+1}}{w_1}E_{j1} - \frac{w_{j+1}}{w_1^2}\alpha_j, &&(j=1,\ldots,N-1),\\
        G_{ij} &= -\frac{w_{j+1}}{w_i}E_{ji} + \frac{\delta_{Ni}\delta_{(N-1)j}}{w_N}, &&(i=2,\ldots,N;j=1,\ldots,N-1).
    \end{aligned} 
    \end{equation}
    Comparing Eq.~\eqref{G} with Eq.~\eqref{Ddag}, it can be concluded that $\m{D}\sdag=\m{G}$. Thus, $\m{D}\sdag\in\R^{N\times (N-1)}$ is a differentiation matrix for the space of polynomials of degree at most $N-2$.
\end{proof}

Equation~\eqref{pseudoderiv transformed adj} approximates the derivative of the costate $\dot{\g{\lambda}}(\tau_i)\approx \m{D}\sdag_{(i,:)}\g{\Lambda},~(i=1,\ldots,N)$, at the $N$ LGL points, assuming the discrete transversality condition $\g{\Lambda}(+1)=\nabla \Phi(\m{X}_N)$ is satisfied. If $\g{\Lambda}(+1)=\nabla \Phi(\m{X}_N)$ is satisfied, then the continuous transversality condition of Eq.~\eqref{1st order opt trans} is satisfied. Furthermore, satisfying the condition $\g{\mu}=\g{\Lambda}_1$ implies that the continuous costate $\g{\lambda}(-1)$ is free, exactly as in the continuous optimality conditions given in Eqs.~\eqref{1st order opt x}-\eqref{1st order opt u}. 

\begin{theorem}
    The condition $\g{\mu}=\g{\Lambda}_1$ is satisfied if and only if the condition $\g{\Lambda}(+1)=\nabla \Phi(\m{X}_N)$ is satisfied. Furthermore, $-\nabla_X\left\langle \g{\Lambda,\m{F}} \right\rangle$ is a discrete approximation of the costate dynamics $\dot{\g{\Lambda}}$ if $\g{\mu}=\g{\Lambda}_1$ and $\g{\Lambda}(+1)=\nabla \Phi(\m{X}_N)$ are satisfied.
\end{theorem}
\begin{proof}
    Suppose $-\nabla_X\left\langle \g{\Lambda,\m{F}} \right\rangle$ of Eq.~\eqref{pseudoderiv transformed adj} approximates the costate dynamics. Furthermore, the costate at the terminal endpoint, $\g{\Lambda}_N\approx\g{\lambda}(+1)$, can be computed via quadrature approximation by
    \begin{equation}\label{Lambda quad approx 1}
        \g{\Lambda}_N = \g{\Lambda}_1 + \sum_{i=1}^{N}w_i\dot{\g{\Lambda}}_i,
    \end{equation}
    where $\g{\Lambda}_1\approx\g{\lambda}(-1)$ approximates the costate at the initial endpoint,  $\dot{\g{\Lambda}}\in\R^{N\times n_x}$ are the discrete costate dynamics, and Eq.~\eqref{Lambda quad approx 1} is the LGL quadrature approximation of the integral of the costate dynamics across the entire interval. Substituting $\dot{\g{\Lambda}}=-\nabla_X\left\langle \g{\Lambda,\m{F}} \right\rangle$ and using the costate expressions from Eq.~\eqref{Lambda pseudoderiv}, Eq.~\eqref{Lambda quad approx 1} can be rewritten as
    \begin{equation}\label{Lambda quad approx 2}
        \g{\Lambda}_N = \g{\Lambda}_1 - \nabla_{X_1}\left\langle \m{S},\g{\alpha\m{F}_1} \right\rangle - \m{1}\T \nabla_{X_{2:N}} \left\langle \m{S},\m{F}_{2:N} \right\rangle.
    \end{equation}
    Next, substituting the KKT conditions from Eqs.~\eqref{pseudoderiv KKT split x1}-\eqref{pseudoderiv KKT split x2thruN} into Eq.~\eqref{Lambda quad approx 2} yields
    \begin{equation}\label{Lambda quad approx 3}
        \g{\Lambda}_N = \g{\Lambda}_1 - \g{\mu} - \m{E}_{(:,1)}\T\m{S} - \m{1}\T \m{E}_{(:,2:N)}\T\m{S} + \nabla \Phi(\m{X}_N).
    \end{equation}
    By Lemma~\ref{lemma E1=0}, $(\m{1}\T\m{E}\T)\m{S}=\m{0}$ and Eq.~\eqref{Lambda quad approx 3} can be written simply as
    \begin{equation}\label{Lambda quad approx 4}
        \g{\Lambda}_N = \g{\Lambda}_1 - \g{\mu}  + \nabla \Phi(\m{X}_N).
    \end{equation}
    Thus, if $\g{\mu}=\g{\Lambda}_1$, then Eq.~\eqref{Lambda quad approx 4} becomes $\g{\Lambda}_N =  \nabla \Phi(\m{X}_N)$. Alternatively, if $\g{\Lambda}_N =  \nabla \Phi(\m{X}_N)$, then Eq.~\eqref{Lambda quad approx 4} becomes $\g{\mu}=\g{\Lambda}_1$. This completes the proof.
\end{proof}

In practice, if the continuous costate solution is smooth but is not a polynomial, $\g{\Lambda}(+1) \approx \nabla \Phi(\m{X}_N)$ and $\g{\mu}\approx\g{\Lambda}_1$ as $N$ is increased. As the expressions $(\g{\mu}-\g{\Lambda}_1)$ and $(\g{\Lambda}_N - \nabla \Phi(\m{X}_N))$ both approach zero, then $\m{D}\sdag\g{\Lambda}_{2:N}=-\nabla_X \left\langle \g{\Lambda},\m{F} \right\rangle = \dot{\g{\Lambda}}$ becomes an increasingly better approximation of the continuous costate dynamics $\dot{\g{\lambda}}$ in Eq.~\eqref{1st order opt lambda}. Lastly, the transformed KKT condition
\begin{equation}
    \m{0} = \nabla_U \left\langle \g{\Lambda}, \m{F}\right\rangle,
\end{equation}
maps to the continuous optimal control law given by Eq.~\eqref{1st order opt u}.

\section{Connections Between Integral Form and Derivative-Like Form}\label{sect:LGL Connections}

The integral form and derivative-like form of the LGL collocation method described in this paper are related because they are based on the same polynomial approximation of the differential equation vector field, where the derivative-like form is obtained by multiplying the integral form by the inverse of a nonsingular $(N-1)\times(N-1)$ matrix.  Given this relationship between the two forms, this section describes some of the key connections between the integral form and derivative-like form of the LGL  collocation method of this paper.

\subsection{Relationship Between the Transformed Adjoint Systems}\label{subsect: transformed adjoint relation}
The discrete costate dynamics given in Eq.~\eqref{integral transformed adj}, which are derived from the integral form of Section~\ref{sect:LGL int}, are exactly the same as those given in Eq.~\eqref{pseudoderiv transformed adj}, which are derived from the derivative-like form of Section~\ref{sect:LGL deriv}. Because $\m{D}\sdag$ is a differentiation matrix for the space of polynomials of degree $N-2$ or less, it can be shown that $\m{A}\sdag$ in Eq.~\eqref{integral transformed adj} is an integration matrix for the space of polynomials of degree $N-3$ or less. 

First, the transformed adjoint system of the integral form in Eq.~\eqref{integral transformed adj} can be written equivalently by substituting $\g{\Lambda}_1$ for an expression in terms of $\g{\Lambda}_{2:N}$ using Eq.~\eqref{Lambda pseudoderiv} and replacing the discrete approximation of the costate dynamics with $\m{D}\sdag\g{\Lambda}_{2:N}$:
\begin{equation}\label{integral transformed adj rewritten}
    \begin{bmatrix}
        \g{\Lambda}_1 \\ \g{\Lambda}_{2:N}
    \end{bmatrix} = 
    \begin{bmatrix}
        \g{\Lambda}_1 \\ \ds \frac{1}{w_1}\m{1}\g{\alpha}\T\m{W}_{2:N}\g{\Lambda}_{2:N}
    \end{bmatrix} + 
    \begin{bmatrix}
        \m{A}\sdag_{(1,:)}\m{D}\sdag\g{\Lambda}_{2:N} \\
        \m{A}\sdag_{(2:N,:)}\m{D}\sdag\g{\Lambda}_{2:N}
    \end{bmatrix}.
\end{equation}

\begin{proposition}\label{proposition AD=0}
    $\m{A}\sdag_{(1,:)}\m{D}\sdag = \m{0}$
\end{proposition}
\begin{proof}
    In order for $\m{A}\sdag_{(1,:)}\m{D}\sdag=\m{0}$ to be true, it must be shown that
    \begin{equation}\label{prop AD=0 general}
        \sum_{j=1}^N A\sdag_{1j}D\sdag_{ji} = 0, ~\forall i\in\{1,\ldots,N-1\}.
    \end{equation}
    
    For $i=1,\ldots,N-2$, the left-hand side of Eq.~\eqref{prop AD=0 general} is
    \begin{equation}\label{prop AD=0 (a)}
    \begin{aligned}
        \sum_{j=1}^N A\sdag_{1j}D\sdag_{ji} &= w_1\left(1-\frac{A_{11}}{w_1}\right)\cdot\frac{1}{w_1}\left( -w_{i+1} E_{i1} - \frac{w_{i+1}}{w_1}\alpha_i \right) + \sum_{j=2}^N \left[ w_j \left( 1-\frac{A_{j1}}{w_1} \right) \cdot \frac{1}{w_j}\left( -w_{i+1}E_{ij} \right) \right],
    \end{aligned}
    \end{equation}
    where the elements of $\m{A}\sdag$ and $\m{D}\sdag$ are defined in Eq.~\eqref{Adag} and \eqref{Ddag}, respectively. Note, $A_{11}=0$. Then, Eq.~\eqref{prop AD=0 (a)} can be written more compactly as
    \begin{equation}\label{prop AD=0 (b)}
        \sum_{j=1}^N A\sdag_{1j}D\sdag_{ji} = -w_{i+1}E_{i1} - \frac{w_{i+1}}{w_1}\alpha_i + \left(-w_{i+1}\m{E}_{(i,2:N)}\right)\left( \m{1} - \frac{1}{w_1}\m{A}_{(2:N,1)} \right).
    \end{equation}
    By Lemma~\ref{lemma E1=0}, $-\m{E}_{(i,2:N)}\m{1}=E_{i1}$. Therefore, right side of Eq.~\eqref{prop AD=0 (b)} simplifies to 
    \begin{equation}
    \begin{aligned}
        \sum_{j=1}^N A\sdag_{1j}D\sdag_{ji} &= -\frac{w_{i+1}}{w_1}\left( {\alpha}_i - \m{E}_{(i,2:N)}\m{A}_{(2:N,1)} \right) \\
        &= -\frac{w_{i+1}}{w_1}\left( \tilde{\m{e}}_i\T \left[ \m{A}_{(2:N,2:N)}\right]\inv - \m{E}_{(i,2:N)} \right) \m{A}_{(2:N,1)},
    \end{aligned}
    \end{equation}
    where $\alpha_i=\tilde{\m{e}}_i\T \g{\alpha}$ following the definition of $\g{\alpha}$ in Eq.~\eqref{alpha}, and $\tilde{\m{e}}_i\T$ is the $i^{\tx{th}}$ row of $\m{I}_{N-1}$. By Lemma~\ref{lemma EA=I}, $\tilde{\m{e}}_i\T [ \m{A}_{(2:N,2:N)}]\inv = \m{E}_{(i,2:N)}$ and $\m{A}\sdag_{(1,:)}\m{D}\sdag_{(:,1:N-2)}=\m{0}$.

    Finally, for $i=N-1$, the proof follows very similarly as for $i=1,\ldots,N-2$. Therefore,
    \begin{equation}
    \begin{aligned}
        \sum_{j=1}^N A\sdag_{1j}D\sdag_{j,N-1} &= -w_{N}E_{N-1,1} - \frac{w_{N}}{w_1}\alpha_{N-1} + \left(-w_{N}\m{E}_{(N-1,2:N)}+\tilde{\m{e}}_{N-1}\T\right)\left( \m{1} - \frac{1}{w_1}\m{A}_{(2:N,1)} \right)
        \\&= -\frac{w_{N}}{w_1}\left( {\alpha}_{N-1} - \m{E}_{(N-1,2:N)}\m{A}_{(2:N,1)} \right) + \tilde{\m{e}}_{N-1}\T\m{1} - \frac{1}{w_1}\tilde{\m{e}}_{N-1}\T \m{A}_{(2:N,1)}\\
        &= -\frac{w_{N}}{w_1}\left( \tilde{\m{e}}_{N-1}\T \left[ \m{A}_{(2:N,2:N)}\right]\inv - \m{E}_{(N-1,2:N)} \right) \m{A}_{(2:N,1)} \\
        &= 0,
    \end{aligned}
    \end{equation}
    where $\tilde{\m{e}}_{N-1}\T\m{1}=1$ and $\tilde{\m{e}}_{N-1}\T \m{A}_{(2:N,1)}=w_1$ because $A_{Ni}=w_i,~(i=1,\ldots,N)$. This completes the proof.
\end{proof}

\begin{proposition}\label{proposition costate identity}
    $\ds \frac{1}{w_1}\m{1}\g{\alpha}\T\m{W}_{2:N} + \m{A}\sdag_{(2:N,:)}\m{D}\sdag = \m{I}_{N-1}$.
\end{proposition}
\begin{proof}
    The expression in the statement of the proposition can be written equivalently as
    \begin{equation}\label{prop costate iden (a)}
        \begin{bmatrix} \ds
            \frac{1}{w_1}\m{1}\g{\alpha}\T\m{W}_{2:N} + \m{A}\sdag_{(2:N-1,:)}\m{D}\sdag \\ \ds
            \frac{1}{w_1}\g{\alpha}\T\m{W}_{2:N} + \m{A}\sdag_{(N,:)}\m{D}\sdag
        \end{bmatrix} = \begin{bmatrix}
            \m{I}_{N-2} & \m{0} \\ \m{0} & 1
        \end{bmatrix}.
    \end{equation}
    Following Eq.~\eqref{Adag}, $\m{A}\sdag_{(2:N,:)}$ can be written compactly as
    \begin{equation}
        \m{A}\sdag_{(2:N,:)} = \begin{bmatrix}
            w_1\m{1} & \left( \m{11}\T - \left[ \m{W}_{2:N-1} \right]\inv \m{A}_{(2:N-1,2:N-1)}\T  \right)\m{W}_{2:N-1} & w_N \left( \m{1} - \left[ \m{W}_{2:N-1} \right]\inv \m{A}_{(N,2:N-1)}\T  \right) \\
            w_1 & \m{1}\T \m{W}_{2:N-1} & w_N
        \end{bmatrix}.
    \end{equation}
    Next, following Eq.~\eqref{Ddag}, $\m{D}\sdag$ can be written compactly as
    \begin{equation}
        \m{D}\sdag = \begin{bmatrix} \ds
            -\frac{1}{w_1}\left( \m{E}_{(:,1)}\T + \frac{1}{w_1}\g{\alpha}\T \right) \m{W}_{2:N} \\
            -\left[ \m{W}_{2:N-1} \right]\inv \m{E}_{(:,2:N-1)}\T \m{W}_{2:N} \\ \ds
            -\frac{1}{w_N}\left( \m{E}_{(:,N)}\T - \frac{1}{w_N}\tilde{\m{e}}_{N-1}\T \right) \m{W}_{2:N}
        \end{bmatrix}.
    \end{equation}

    Now, the left-hand side of the first row of Eq.~\eqref{prop costate iden (a)} is
    \begin{equation}\label{prop costate iden (b)}
    \begin{aligned}
        \frac{1}{w_1}\m{1}\g{\alpha}\T\m{W}_{2:N} + \m{A}\sdag_{(2:N-1,:)}\m{D}\sdag &= -\m{1} \left( \m{1}\T \m{E}\T \right) \m{W}_{2:N} + \left[ \m{W}_{2:N-1} \right]\inv \left( \m{A}_{(2:N,2:N-1)}\T \m{E}_{(:,2:N)}\T \right) \m{W}_{2:N} \\ &~\quad + \frac{1}{w_N}\m{1}\tilde{\m{e}}_{N-1}\T \m{W}_{2:N} - \frac{1}{w_N} \left[ \m{W}_{2:N-1} \right]\inv \m{A}_{(N,2:N-1)}\T \tilde{\m{e}}_{N-1}\T \m{W}_{2:N}.
    \end{aligned}
    \end{equation}
    By Lemma~\ref{lemma E1=0}, $\m{1}\T\m{E}\T=\m{0}$. Additionally, $\m{A}_{(N,2:N-1)}\T = \m{W}_{2:N-1}\m{1}$. Lastly, Lemma~\ref{lemma EA=I} implies
    \begin{equation}
    \m{A}_{(2:N,2:N-1)}\T \m{E}_{(:,2:N)}\T = \begin{bmatrix} \m{I}_{N-2} & \m{0} \end{bmatrix}.    
    \end{equation}
    Therefore, Eq.~\eqref{prop costate iden (b)} simplifies to
    \begin{equation}\label{prop costate iden (c)}
    \begin{aligned}
        \frac{1}{w_1}\m{1}\g{\alpha}\T\m{W}_{2:N} + \m{A}\sdag_{(2:N-1,:)}\m{D}\sdag &= \left[ \m{W}_{2:N-1} \right]\inv \begin{bmatrix}
            \m{I}_{N-2} & \m{0}
        \end{bmatrix} \begin{bmatrix}
            \m{W}_{2:N-1} & \m{0} \\ \m{0} & w_N
        \end{bmatrix}\\
        &= \begin{bmatrix}
            \m{I}_{N-2} & \m{0}
        \end{bmatrix}.
    \end{aligned}
    \end{equation}

    The left-hand side of the second row of Eq.~\eqref{prop costate iden (a)} simplifies to
    \begin{equation}\label{prop costate iden (d)}
    \begin{aligned}
        \frac{1}{w_1}\g{\alpha}\T\m{W}_{2:N} + \m{A}\sdag_{(N,:)}\m{D}\sdag &= -\left( \m{1}\T \m{E}\T \right) \m{W}_{2:N} + \frac{1}{w_N}\tilde{\m{e}}_{N-1}\T \m{W}_{2:N} \\
        &= \tilde{\m{e}}_{N-1}\T.
    \end{aligned}
    \end{equation}
    Finally, Eqs.~\eqref{prop costate iden (c)}-\eqref{prop costate iden (d)} can be combined to conclude that $\ds \frac{1}{w_1}\m{1}\g{\alpha}\T\m{W}_{2:N} + \m{A}\sdag_{(2:N,:)}\m{D}\sdag = \m{I}_{N-1}$.
\end{proof}

The results of Propositions~\ref{proposition AD=0} and \ref{proposition costate identity} lead to the following theorem regarding the transformed adjoint system of the integral form in Section~\ref{subsect: transformed KKT int}.

\begin{theorem}
    If the transformed adjoint system of Eq.~\eqref{integral transformed adj rewritten} is satisfied, then $\m{A}\sdag$ is an integration matrix for the space of polynomials of degree at most $N-3$. More precisely, if \[\m{A}\sdag_{(1,:)}\m{D}\sdag = \m{0}\] and \[\ds \frac{1}{w_1}\m{1}\g{\alpha}\T\m{W}_{2:N} + \m{A}\sdag_{(2:N,:)}\m{D}\sdag = \m{I}_{N-1},\] then $\m{A}\sdag$ must be an integration matrix for the space of polynomials of degree $N-3$ or less.
\end{theorem}
\begin{proof}
    By Theorem~\ref{theorem Ddag}, $\m{D}\sdag$ is a differentiation matrix for the space of polynomials of degree at most $N-2$. Therefore, $\m{D}\sdag \g{\Lambda}_{2:N}$ approximates the costate dynamics as a polynomial of degree at most $N-3$. If the results of Propositions~\ref{proposition AD=0} and \ref{proposition costate identity} are substituted into Eq.~\eqref{integral transformed adj rewritten}, then $\g{\Lambda} = \m{1}\g{\Lambda}_1 + {\m{A}}\sdag\m{D}\sdag \g{\Lambda}_{2:N}$ is satisfied which concludes the proof.
\end{proof}

Additionally, $\m{A}\sdag\in\R^{N\times N}$ is full rank and, thus, invertible. Therefore, Eq.~\eqref{integral transformed adj} can be written equivalently as 
\begin{equation}\label{integral to deriv transformed adj}
    \m{D}\sddag \g{\Lambda} = -\nabla_{X}\left\langle \g{\Lambda}, \m{F}\right\rangle + \frac{\m{e}_1}{w_1}(\g{\mu} - \g{\Lambda}_1)+\frac{\m{e}_N}{w_N}(\g{\Lambda}_N - \nabla \Phi(\m{X}_N)), 
\end{equation}
where
\begin{equation}\label{Ddag square}
    \quad \m{D}\sddag := 
    [\m{A}\sdag]\inv \begin{bmatrix}
        0 & \m{0} \\ -\m{1} & \m{I}_{N-1}
    \end{bmatrix} \in \R^{N\times N},
\end{equation}
is a differentiation matrix for the space of polynomials of degree at most $N-2$. 

Both $\m{D}\sdag$ and $\m{D}\sddag$ operate on polynomial values to give the derivative at the collocation points. However, $\m{D}\sdag$ operates on the polynomial values $y(\tau_i),~2\leq i \leq N$, while $\m{D}\sddag$ operates on the polynomial values $y(\tau_i),~1\leq i \leq N$, where $y$ is a polynomial of degree at most $N-2$.
Because the first row of $\m{A}\sdag$ is not zero, the first equation of the transformed adjoint system in Eq.~\eqref{integral transformed adj} is not omitted as was done for the state dynamics in Eq.~\eqref{LGL integral form}. In fact, if the costate solution is a polynomial of degree at most $N-2$, then $\m{A}\sdag_{(1,:)}\dot{\g{\Lambda}}=\m{0}$ and Eq.~\eqref{integral transformed adj} is satisfied. Because the first equation of the transformed adjoint system in Eq.~\eqref{integral transformed adj} is not omitted, inversion of $\m{A}\sdag$ leads to the equivalent transformed adjoint system in Eq.~\eqref{integral to deriv transformed adj}, where it is seen that $\m{D}\sddag$ operates on the costate values at all the LGL points. However, when the LGL integral form is transformed into the derivative-like form, the transformed adjoint system can be expressed using a differentiation matrix that operates on $\g{\Lambda}_{2:N}$ because ${\Lambda}_{(1,m)}$ is a linear combination of all the NLP multipliers $\m{S}_{(:,m)},~m\in\{1,\ldots,n_x\}$. Therefore, it is possible to express the transformed adjoint system in Eq.~\eqref{pseudoderiv transformed adj} using an $N\times N$ differentiation matrix, but this matrix would not be unique. That is, given the costate approximation defined in Eq.~\eqref{Lambda pseudoderiv} and some differentiation matrix $\Tilde{\m{D}}\in\R^{N\times N}$ such that
\begin{equation}\label{nonunique transformed adjoint}
    \Tilde{\m{D}} \g{\Lambda} = -\nabla_{X}\left\langle \g{\Lambda}, \m{F}\right\rangle + \frac{\m{e}_1}{w_1}(\g{\mu} - \g{\Lambda}_1)+\frac{\m{e}_N}{w_N}(\g{\Lambda}_N - \nabla \Phi(\m{X}_N)), 
\end{equation}
there exists many solutions to the elements of $\Tilde{\m{D}}$ such that Eq.~\eqref{nonunique transformed adjoint} maps to the KKT conditions in Eqs.~\eqref{pseudoderiv KKT split x1} and \eqref{pseudoderiv KKT split x2thruN}. Because $\m{D}\sdag$ is a differentiation matrix for the space of polynomials of degree at most $N-2$, the costate $\g{\lambda}$ is approximated as a polynomial of degree at most $N-2$. Moreover, $N-1$ points are necessary to uniquely define a polynomial of degree $N-2$. Therefore, the transformed adjoint system of Eq.~\eqref{pseudoderiv transformed adj} which employs $\m{D}\sdag$ implies that the costate can be approximated as a polynomial of degree $N-2$ using the discrete costate approximations at $(\tau_2,\ldots,\tau_N)$, where $\g{\lambda}(\tau_i)\approx\g{\Lambda}_i$. 

\subsection{Relationship Between the NLP Multipliers}
The NLP multipliers $\m{M}$ that are obtained from solving the NLP associated with the LGL integral form of Section~\ref{sect:LGL int} are not the same as the NLP multipliers $\m{S}$ that are obtained from solving the NLP associated with the LGL derivative-like form of Section~\ref{sect:LGL deriv}. However, they are related via a linear transformation. 
\begin{proposition}\label{proposition multipliers}
    The Lagrange multipliers $\m{M}\in\R^{(N-1)\times n_x}$ and $\m{S}\in\R^{(N-1)\times n_x}$ are related to one another by $\m{S}=[\tilde{\m{A}}_{(:,2:N)}]\T \m{M}$, assuming that $(\m{M} - [\tilde{\m{A}}_{(:,2:N)}]^{-{\sf T}} \m{S})$ does not lie in the orthogonal complement of $( \begin{bmatrix} \m{1} & -\m{I}_{N-1} \end{bmatrix} \m{X} + \tilde{\m{A}}\m{F})$ under the Frobenius inner product.
\end{proposition}
\begin{proof}
The Lagrangian of the NLP associated with the LGL integral form, given by Eq.~\eqref{integral Lagrangian}, is equal to the Lagrangian of the NLP associated with the LGL derivative-like form, given by Eq.~\eqref{pseudoderiv Lagrangian}. Setting them equal to one another results in
\begin{equation}\label{prop multiplier equality}
    \left\langle \m{M}, \begin{bmatrix} \m{1} & -\m{I}_{N-1} \end{bmatrix} \m{X} + \tilde{\m{A}}\m{F} \right\rangle = \left\langle \m{S}, \left[\tilde{\m{A}}_{(:,2:N)}\right]\inv \left( \begin{bmatrix} \m{1} & -\m{I}_{N-1} \end{bmatrix} \m{X} + \tilde{\m{A}}\m{F} \right) \right\rangle,
\end{equation}
where 
\begin{equation}
    \left[\tilde{\m{A}}_{(:,2:N)}\right]\inv \left( \begin{bmatrix} \m{1} & -\m{I}_{N-1} \end{bmatrix} \m{X} + \tilde{\m{A}}\m{F} \right) \equiv \begin{bmatrix}
        \g{\alpha} & \m{I}_{N-1}
    \end{bmatrix}\m{F} - \m{EX}.
\end{equation}
Using the properties that $\langle \m{P},\m{Q}\rangle = \tx{tr}(\m{P}\T\m{Q})$ and $\tx{tr}(\m{P}\T\m{Q}) = \tx{tr}(\m{Q}\T\m{P})$ for matrices $\m{P},\m{Q}\in\R^{m\times n}$, Eq.~\eqref{prop multiplier equality} can be rewritten as
\begin{equation}\label{prop multiplier trace}
    \tx{tr}\left( \left( \begin{bmatrix} \m{1} & -\m{I}_{N-1} \end{bmatrix} \m{X} + \tilde{\m{A}}\m{F} \right)\T \m{M} \right) = \tx{tr}\left( \left( \begin{bmatrix} \m{1} & -\m{I}_{N-1} \end{bmatrix} \m{X} + \tilde{\m{A}}\m{F} \right)\T \left[\tilde{\m{A}}_{(:,2:N)}\right]^{-{\sf T}} \m{S} \right).
\end{equation}
Therefore, $\m{M} = [\tilde{\m{A}}_{(:,2:N)}]^{-{\sf T}} \m{S}$; equivalently, $\m{S}=[\tilde{\m{A}}_{(:,2:N)}]\T \m{M}$.
\end{proof}

\section{Second Integral Form of LGL Collocation}\label{sect: arxiv comparison and new results}

In this section, a second integral form of LGL collocation is derived from the integral form of LGL collocation developed in this paper (see Section \ref{sect:LGL int}).  This second integral form includes an additional integral, and it is shown that the inclusion of this additional integral is superfluous with regard to the form of the LGL NLP developed in this paper.  Next, it is shown that, while superfluous from the perspective of the NLP, the inclusion of the additional integral leads to a second integral form which is equivalent to the differential form described in Ref.~\cite{Garrido2023}. Finally, it is shown that this additional integral can be used to generate a polynomial approximation of degree $N$ for the state.  
    
\subsection{Second LGL Integral Form Using Method of Section \ref{sect:LGL int}}\label{section:second integral form}

Consider the integral of the state dynamics from $\tau=-1$ to point $\tau=\tau_{N+1}$, where $\tau_{N+1}$ is \emph{not} an LGL point.  Furthermore, let $\m{X}_{N+1}\approx\m{x}(\tau_{N+1})$ be the approximation of the state at $\tau=\tau_{N+1}$.  Recalling that the differential equation vector field is approximated as a polynomial of degree at most $N-1$ as given by Eq.~\eqref{f approx}, the integral from $\tau=-1$ to $\tau=\tau_{N+1}$ is given as
\begin{equation}\label{x superfluous}
    \m{X}_{N+1} = \m{X}_{1} + \sum_{j=1}^N \int_{-1}^{\tau_{N+1}} \m{F}_jL_j(\tau)\dt{\tau}
\end{equation}
Because $L_j(\tau),\;(i=1,\ldots,N)$ are polynomials of degree $N-1$, the integral in Eq.~\eqref{x superfluous} can be evaluated exactly using LGL quadrature to obtain
\begin{equation}\label{x superfluous LGL}
    \m{X}_{N+1} = \m{X}_{1} + \sum_{j=1}^N A_{(N+1,j)} \m{F}_j = \m{X}_{1} + \m{A}_{(N+1,:)}\m{F}
\end{equation}
where $\m{A}_{(N+1,:)}$ is treated as an additional row concatenated to the end of $\m{A}$ defined in Eq.~\eqref{Axelsson integration matrix elements}. 
Suppose now that $(\m{X}_{1:N},\m{U}_{1:N})$ is the solution to the NLP given in Eq.~\eqref{integral NLP} of Section \ref{sect:LGL int}.  Then, the value $\m{X}_{N+1}$ can be obtained purely from this solution $(\m{X}_{1:N},\m{U}_{1:N})$. Thus, the value $\m{X}_{N+1}$ is superfluous; that is, $\m{X}_{N+1}$ is \emph{not} an optimization variable in the NLP of the integral form of the LGL collocation method described in this paper.  Finally, because the integral form given in Section \ref{sect:LGL int} and the derivative-like form given in Section \ref{sect:LGL deriv} are equivalent, the value $\m{X}_{N+1}$ is superfluous to either of these NLPs.  

Although $\m{X}_{N+1}$ in Eq.~\eqref{x superfluous} is superfluous, a second integral LGL method can be constructed by augmenting Eq.~\eqref{x superfluous} to the system given in Eq.~\eqref{LGL integral form} to obtain the system
\begin{equation}\label{superfluous integral form}
\begin{bmatrix}
    \m{X}_{2:N} \\ \m{X}_{N+1}
\end{bmatrix} = \begin{bmatrix}
    \m{1}\m{X}_1 \\ \m{X}_1
\end{bmatrix} + \begin{bmatrix}
    \tilde{\m{A}}\m{F} \\ \m{A}_{(N+1,:)}\m{F}
\end{bmatrix}.
\end{equation}
Suppose now that the matrix $\m{B}$ is defined as
\begin{equation}\label{superfluous integration matrix B}
\m{B} = 
\left[
    \begin{array}{c}
        \tilde{\m{A}} \\ 
        \m{A}_{(N+1,:)} 
    \end{array}
\right]
\end{equation}
Then Eq.~\eqref{superfluous integral form} can be written as
\begin{equation}\label{LGL int using B}
    \m{X}_{2:N+1} = \m{1}\m{X}_1 + \m{B} \m{F},
\end{equation}
where matrix $\m{B}\in\bb{R}^{N\times N}$ in Eq.~\eqref{LGL int using B} is an integration matrix of rank $N$ whose elements are defined as
\begin{equation}\label{B matrix}
    B_{ij} := \int_{-1}^{\tau_i} L_j(\tau)\dt{\tau}, \quad (i=2,\ldots,N+1;~j=1,\ldots,N).
\end{equation}
The matrix $\m{B}$ is full rank because the integrals from $\tau=-1$ to $\tau=\{\tau_i ~|~ 2 \leq i \leq N+1\}$ are unique. 

\subsection{Equivalence Between Second Integral Form and Derivative Form}

Now, suppose that the differential form developed in Ref.~\cite{Garrido2023} as
\begin{equation}\label{DX=F}
    \m{D}\m{X}_{1:N+1} = \m{F}
\end{equation}
is considered, 
where $\m{D}^{N\times(N+1)}$ as defined in Ref.~\cite{Garrido2023} is a full rank differentiation matrix for the space of polynomials of degree $N$ or less\cite{Garrido2023}.  Now, it was shown in Ref.~\cite{Garrido2023} that $\m{D}\m{1} = \m{0}$. Decomposing $\m{D}$ into its first column, $\m{D}_{(:,1)}$ and its remaining $N$ columns, $\m{D}_{(:,2:N+1)}$, Eq.~\eqref{DX=F} can be written as
\begin{equation}\label{DX=F decomposed}
\begin{aligned}
    \m{D}\m{X}_{1:N+1} &= \begin{bmatrix} \m{D}_{(:,1)} & \m{D}_{(:,2:N+1)}\end{bmatrix} \begin{bmatrix}
        \m{X}_{1} \\ \m{X}_{2:N+1}
    \end{bmatrix} \\ 
    &= \m{D}_{(:,1)}\m{X}_{1} + \m{D}_{(:,2:N+1)}\m{X}_{2:N+1} = \m{F}.
\end{aligned}
\end{equation}
Then, because $\m{D}\m{1}=\m{0}$ \cite{Garrido2023} and $\m{D}$ is full rank \cite{Garrido2023}, $\m{D}_{(:,2:N+1)}$ is invertible and Eq.~\eqref{DX=F decomposed} can be written as
\begin{equation}\label{LGL int using inverse of D2N+1}
    \m{X}_{2:N+1} = \m{1}\m{X}_{1} + \left[\m{D}_{(:,2:N+1)}\right]\inv \m{F} 
\end{equation}
Because Eqs.~\eqref{LGL int using B} and \eqref{LGL int using inverse of D2N+1} are equivalent, the matrix $\m{D}_{(:,2:N+1)}$ has the property that
\begin{equation}
    \left[\m{D}_{(:,2:N+1)}\right]\inv = \m{B}. 
\end{equation}
Consequently, the solutions of the optimal control problem using the differential form of Ref.~\cite{Garrido2023} and the second integral form given in Section \ref{section:second integral form} are equivalent. 

\subsection{Further Analysis of Additional Integral}

Recall that the LGL method developed in this paper approximates the differential equation vector field as a polynomial of degree $N-1$,  but it does not explicitly approximate the state as a polynomial.  Note, however, a polynomial state approximation of degree $N$ can be obtained by employing the value $\m{X}_{N+1}$,
where it is emphasized that $\m{X}_{N+1}$ is not an optimization variable when solving the NLP of either Eq.~\eqref{integral NLP} or \eqref{pseudoderivative NLP} but is instead computed from Eq.~\eqref{x superfluous LGL} using the NLP solution $(\m{X}_{1:N},\m{U}_{1:N})$. 
As stated in Section \ref{section:second integral form}, the value $\m{X}_{N+1}$ in the second integral form is superfluous because it has no effect on the solution to the NLP. Now, because $\m{X}_{N+1}$ is superfluous with regard to the form of the LGL NLP developed in this paper, the particular choice of $\tau_{N+1}$ is arbitrary as long as $\tau_{N+1}$ is distinct from any of the $N$ LGL points, i.e. $\tau_{N+1}\neq\tau_i~\forall i\in\{1,\ldots,N\}$. Thus, a continuous approximation of the state can be constructed using an interpolating polynomial of degree $N$ with the pair of $N+1$ points $(\tau_j,\m{X}_j),~(j=1,\ldots,N+1)$.

Finally, it is important to note that both the integral form of Section \ref{sect:LGL int} and the derivative-like form of Section \ref{sect:LGL deriv} developed in this paper are computationally beneficial relative to either the second integral of Section \ref{section:second integral form} or the differential form of Ref.~\cite{Garrido2023}.  This computational benefit arises because the superfluous point $\m{X}_{N+1}$ does not appear in either the LGL integral form of Section \ref{sect:LGL int} or the LGL derivative-like form of Section \ref{sect:LGL deriv}.  Consequently, the algebraic system using either form developed in this paper contain $n_x$ fewer variables than does either the second integral form of Section \ref{section:second integral form} or the differential form in Ref.~\cite{Garrido2023} because these latter two forms contain the superfluous variable $\m{X}_{N+1}$. 

\section{Numerical Examples}\label{sect:examples}

The LGL collocation method developed in this paper, henceforth referred to as the LGL-Integration (LGL-I) method, is now demonstrated on two examples.  The first example is taken from Ref.~\cite{Garg2010} and has an analytic solution. Because this first example has an analytic solution, it is possible to show the accuracy that can be attained using the LGL method described in this paper. The second example is taken from Ref.~\cite{Bryson1975} and does not have an analytic solution. Both examples are solved using the NLP solver IPOPT \cite{Biegler2008} and all derivatives required by the NLP solver are obtained using the algorithmic differentiation software ADiGator \cite{Weinstein2017}. Finally, all results that follow were obtained using MATLAB R2023b (release 23.2.0.2668659) running on an Apple M3 MacBook Pro running macOS Sequoia Version 15.1.1 with 18 GB of RAM.

\subsection{Example 1: One-Dimensional Initial-Value Optimal Control Problem}\label{subsect: ex1}

Consider the following optimal control problem \cite{Garg2010}:
\begin{equation}\label{example-1}
    \tx{minimize}~ -y(2) \quad \tx{subject to}~ \ddt{y}{t} = \frac{5}{2}(-y + yu - u^2), \quad y(0) = 1.
\end{equation}
The analytic solution to the optimal control problem described in Eq.~\eqref{example-1} is given as
\begin{equation}
    \begin{aligned}
        y^*(t) &= 4/a(t), \\
        u^*(t) &= y^*(t)/2, \\
        \lambda_y^*(t) &= -\frac{\exp(2\ln(a(t)) - 5t/2)}{\exp(-5) + 6 + 9\exp(5)}, \\
        a(t) & = 1 + 3\exp(5t/2).    
    \end{aligned}
\end{equation}
The optimal control problem given in Eq.~\eqref{example-1} was solved using the following different methods: LG, LGR, LGL \cite{Fahroo2001} (LGL), and LGL-Integration (LGL-I), where LGL-I is the method of this paper.  Reference~\cite{Fahroo2001} states that, in practice, the LGL costate needs to be filtered and provides an acknowledgment to the source of this filter.  Consequently, the results for the LGL costate are obtained using the costate estimation method of \cite{Fahroo2001} together with the implementation of the costate filter as found in Appendix~\ref{sect: appendix filter}.\footnote{The exact steps required to implement the digital filter mentioned in Ref.~\cite{Fahroo2001} were provided to the authors by Tom Thorvaldsen of The Charles Stark Draper Laboratory via private communication.  These steps were implemented in the code shown in Appendix~\ref{sect: appendix filter} and are used via permission of Tom Thorvaldsen.} It is emphasized that the costate of the LGL-I method does not need to be filtered, as demonstrated by the following results. The numerical solutions were obtained using the NLP solver IPOPT with a convergence tolerance of $10^{-14}$.  Finally, the exact solution at the collocation points was used as an initial guess.

\subsubsection{Accuracy Analysis Using a Single Interval\label{sect:example-1-single-interval}}

The accuracy of solutions obtained when solving the optimal control problem given in Eq.~\eqref{example-1} is first analyzed for a single interval as a function of the number of collocation points, $N$. The accuracy is assessed by computing the following $L_\infty$ norm errors in the state, control, and costate, respectively, at the collocation points $(\tau_1,\ldots,\tau_N)$ as a function of $N$: 
\begin{equation}\label{abs error}
    \begin{aligned}
        e_y &= \max_{i\in[1,\ldots,N]}|y(\tau_i) - y^*(\tau_i)|, \\
        e_u &= \max_{i\in[1,\ldots,N]}|u(\tau_i) - u^*(\tau_i)|, \\
        e_{\lambda_y} &= \max_{i\in[1,\ldots,N]}|\lambda_y(\tau_i) - \lambda_y^*(\tau_i)|.
    \end{aligned}
\end{equation}

\begin{figure}[t!]
    \centering
    \begin{subfigure}{0.5\textwidth}
        \centering
        \includegraphics[width=\linewidth]{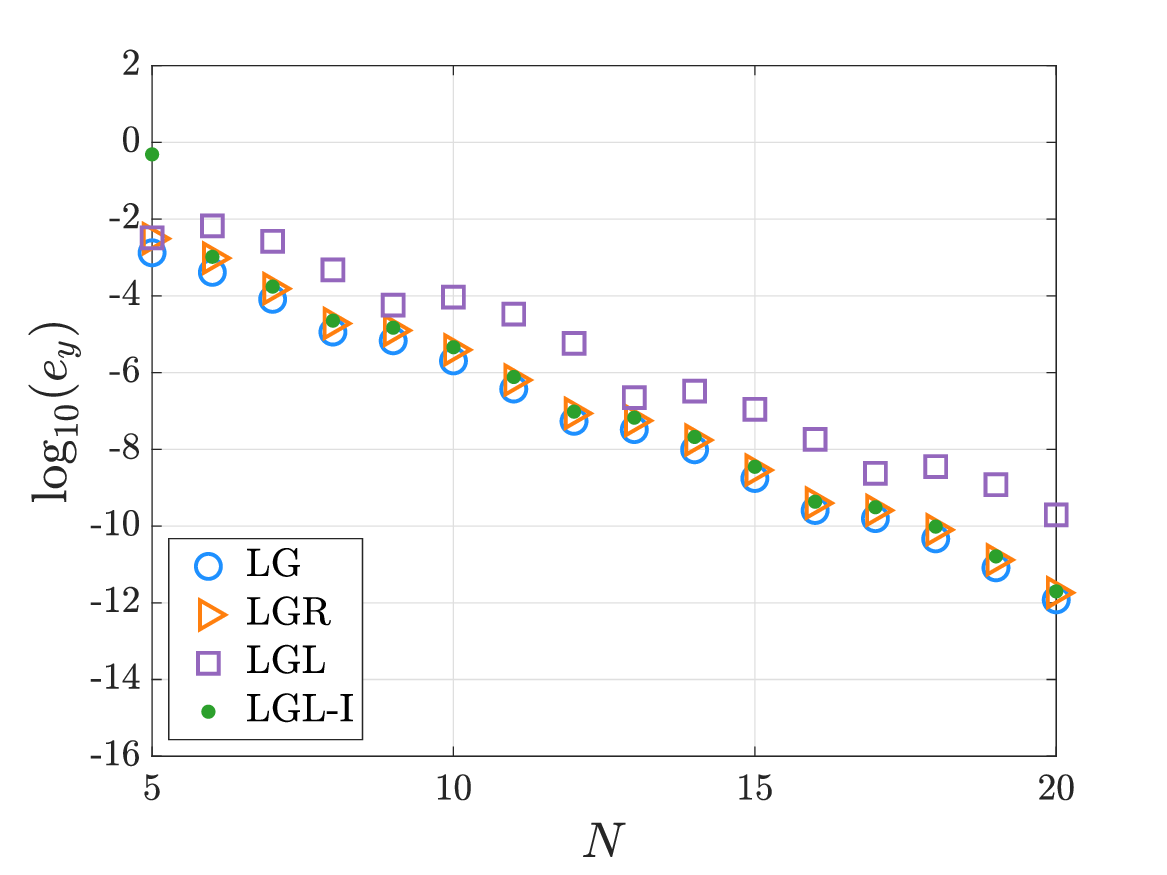}
        \caption{State error.}\label{subfig: 1Divp ey by N}
    \end{subfigure}%
    \begin{subfigure}{0.5\textwidth}
        \centering
        \includegraphics[width=\linewidth]{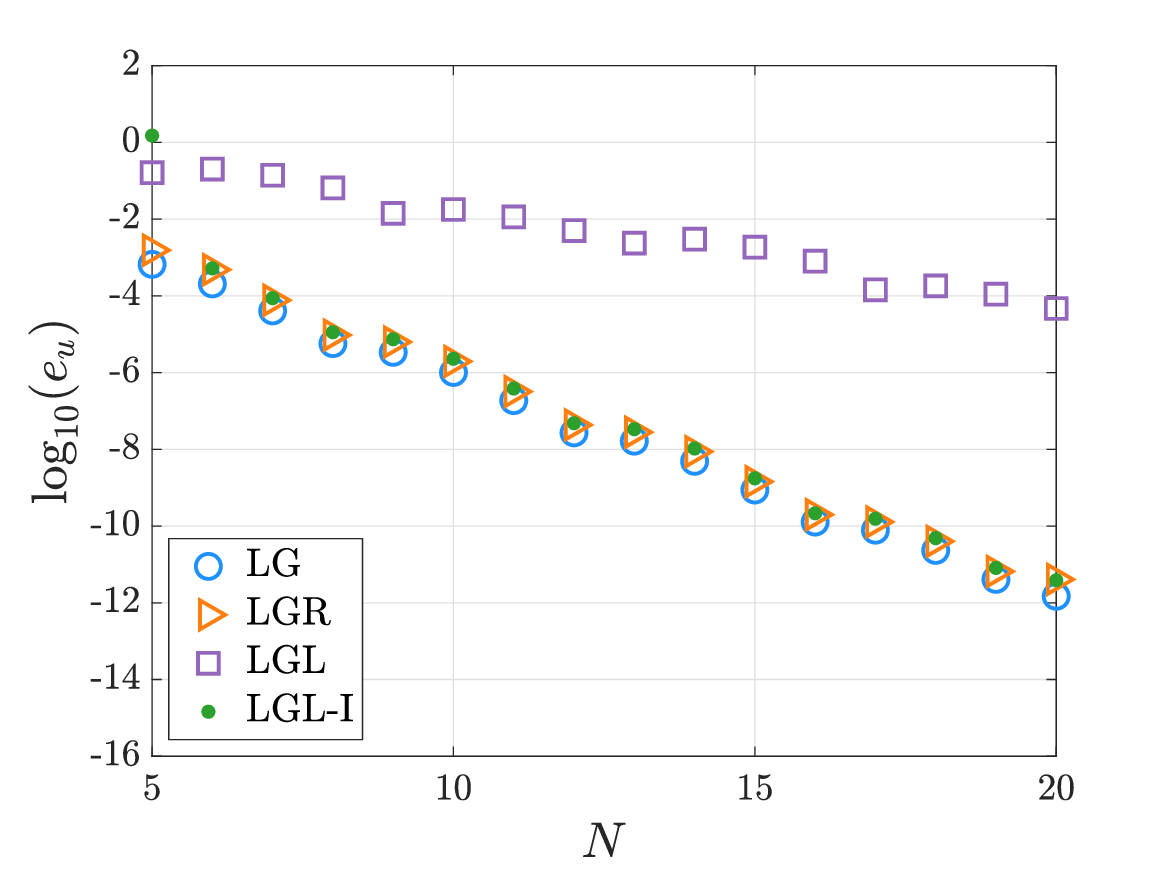}
        \caption{Control error.}\label{subfig: 1Divp eu by N}
    \end{subfigure}\\
    \begin{subfigure}{0.5\textwidth}
        \centering
        \includegraphics[width=\linewidth]{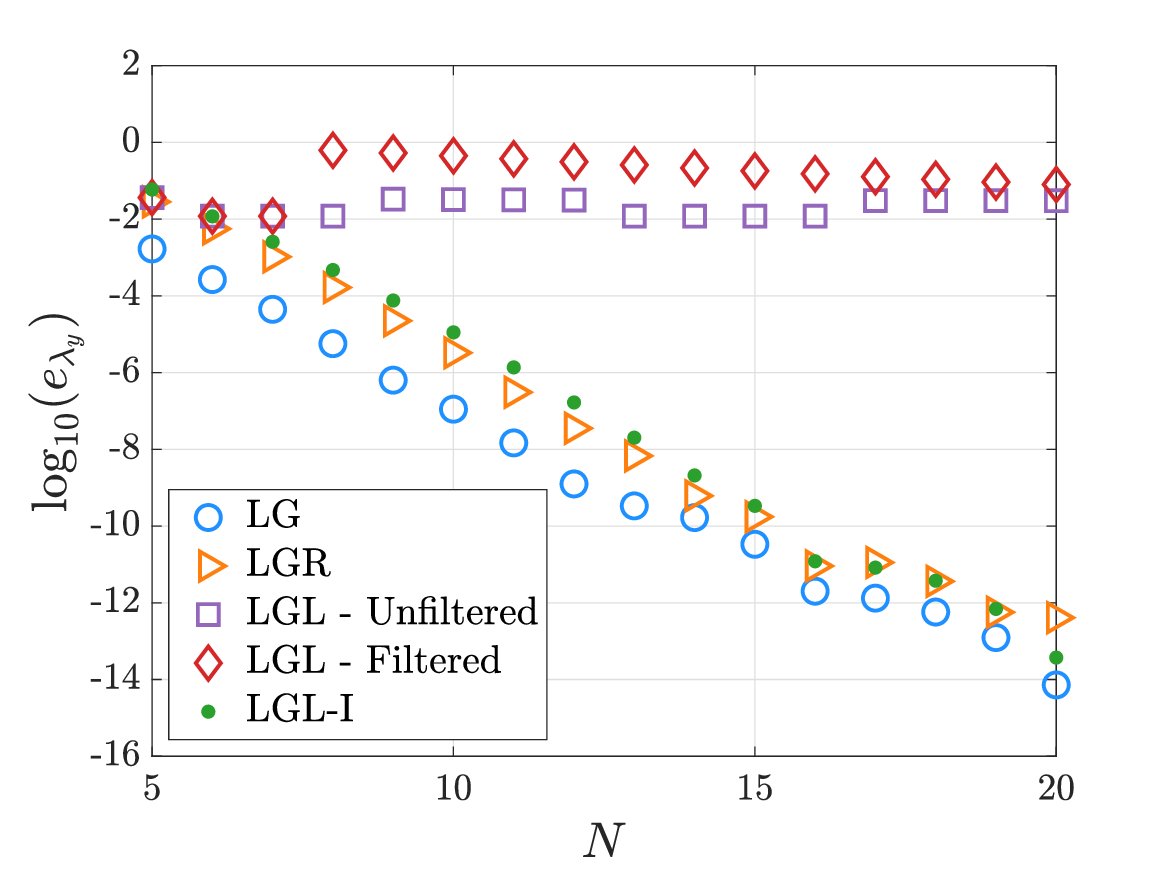}
        \caption{Costate error.}\label{subfig: 1Divp ely by N}
    \end{subfigure}%
    \caption{State, control, and costate errors for Example 1 as a function of the number of LGL points, $N$.}\label{fig: 1Divp errors by N}
\end{figure}

Figure~\ref{fig: 1Divp errors by N} shows the $L_\infty$ norm errors as a function of $N$ for the four aforementioned methods.  It is seen in Fig.~\ref{fig: 1Divp errors by N} that the state, control, and costate errors using the LG, LGR, and LGL-I schemes all decay as $N$ increases.  
The control using the LG, LGR, or LGL-I method is more accurate by at least two orders of magnitude or more compared with the LGL control solution.  More specifically, for $N\leq 10$, the error in the control using the LGL\cite{Fahroo2001} method lies between $10^{-1}$ (at $N=6$) and $10^{-2}$ (at $N=10$), while the error in the control using any of the other methods decreases from $10^{-3}$ (at $N=6$) to $10^{-6}$ (at $N=10$).  Next, it is seen in Fig.~\ref{subfig: 1Divp ely by N} that, for the LG, LGR, and LGL-I methods, the costate error decreases exponentially to near machine precision, while the LGL\cite{Fahroo2001} costate error stays at approximately $\C{O}(1)$ for all values of $N$. The digital filter improves the LGL\cite{Fahroo2001} costate solution as $N$ increases; however, the filtered LGL costate error remains largest near the interval endpoints where the amplitudes of the oscillations in the unfiltered LGL\cite{Fahroo2001} costate are largest, as shown in Fig.~\ref{fig: 1Divp costate}. 

\begin{figure}[htb!]
    \centering
    \includegraphics[width=\linewidth]{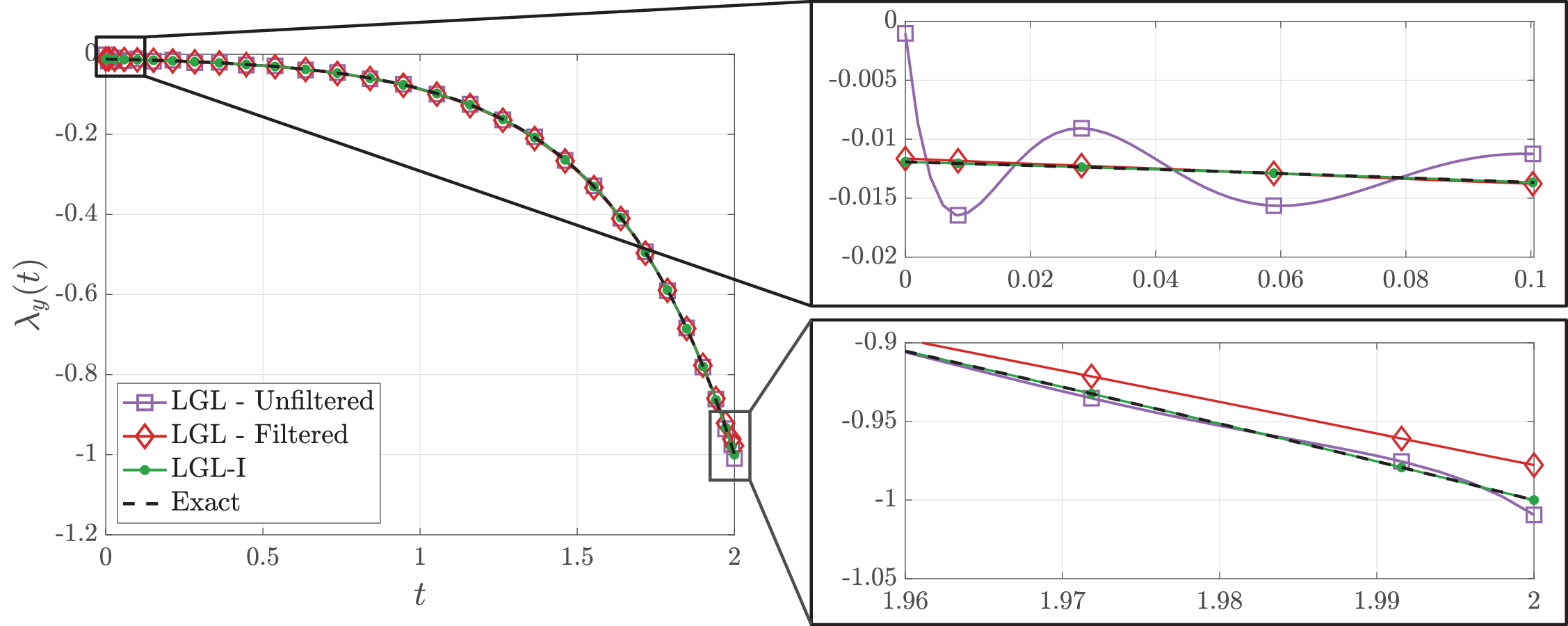}
    \caption{Costate solutions obtained using LGL collocation with $N=30$.}\label{fig: 1Divp costate}
\end{figure}

The nonconvergent behavior of the LGL\cite{Fahroo2001} costate approximation was described for this example in Ref.~\cite{Garg2010} and was shown to be attributed to the nonzero null space of the corresponding transformed adjoint system.  More precisely, the transformed adjoint system of the LGL\cite{Fahroo2001} method represents $N$ equations in $N+1$ unknowns $(\g{\Lambda}_1,\ldots,\g{\Lambda}_N,\g{\mu})$, implying an infinite number of solutions which can manifest in the form of oscillatory behavior in the costate approximation\cite{Garg2010}. Contrariwise, the transformed adjoint system of the LGL-I method of this paper, given by Eq.~\eqref{pseudoderiv transformed adj}, represents $N$ equations in $N$ unknowns $(\g{\Lambda}_2,\ldots,\g{\Lambda}_N,\g{\mu})$ because $\g{\Lambda}_1$ can be expressed in terms of $(\g{\Lambda}_2,\ldots,\g{\Lambda}_N)$.

As discussed in Section~\ref{sect: arxiv comparison and new results}, the inclusion of a noncollocated point in the state approximation is superfluous in the LGL-I method. Moreover, the location of the additional polynomial support point $\tau_{N+1}$ is arbitrary as long as it is not one of the LGL points. The differential form of LGL collocation in Ref.~\cite{Garrido2023} argues that the location of the noncollocated support point needs to be placed optimally in order to minimize Runge phenomenon in the interpolation of the state. In order to analyze the possible numerical impact of the choice of $\tau_{N+1}$, the method of Ref.~\cite{Garrido2023} can be implemented with varying values of $\tau_{N+1}\in(-1,+1)$. Then, the root mean square error (RMSE) of the state can be computed using the state polynomial approximation interpolated to a sample of 1000 equally spaced nodes in $\tau\in[-1,+1]$. Thus, the RMSE takes into consideration not just the error in the state approximation \emph{at} the collocation points, but also the error in the state approximation \emph{in between} the collocation points, where Runge phenomenon would manifest if present. Figure~\ref{fig: 1Divp taustar position} demonstrates that the paricular choice of $\tau_{N+1}$ has no impact on the solution (beyond potential round-off error) to the discretized problem posed by Ref.~\cite{Garrido2023}. Furthermore, for any chosen $\tau_{N+1}$, the state RMSE computed using the state approximation for the LGL-I method closely matches the state RMSE for the method of Ref.~\cite{Garrido2023}. Therefore, the LGL-I method achieves the same order of accuracy in the state approximation as the method of Ref.~\cite{Garrido2023} and equivalent second integral form of this paper, but it does so without needing to include an NLP variable for the state at $\tau_{N+1}$.

\begin{figure}[htb!]
    \centering
    \includegraphics[scale=0.5]{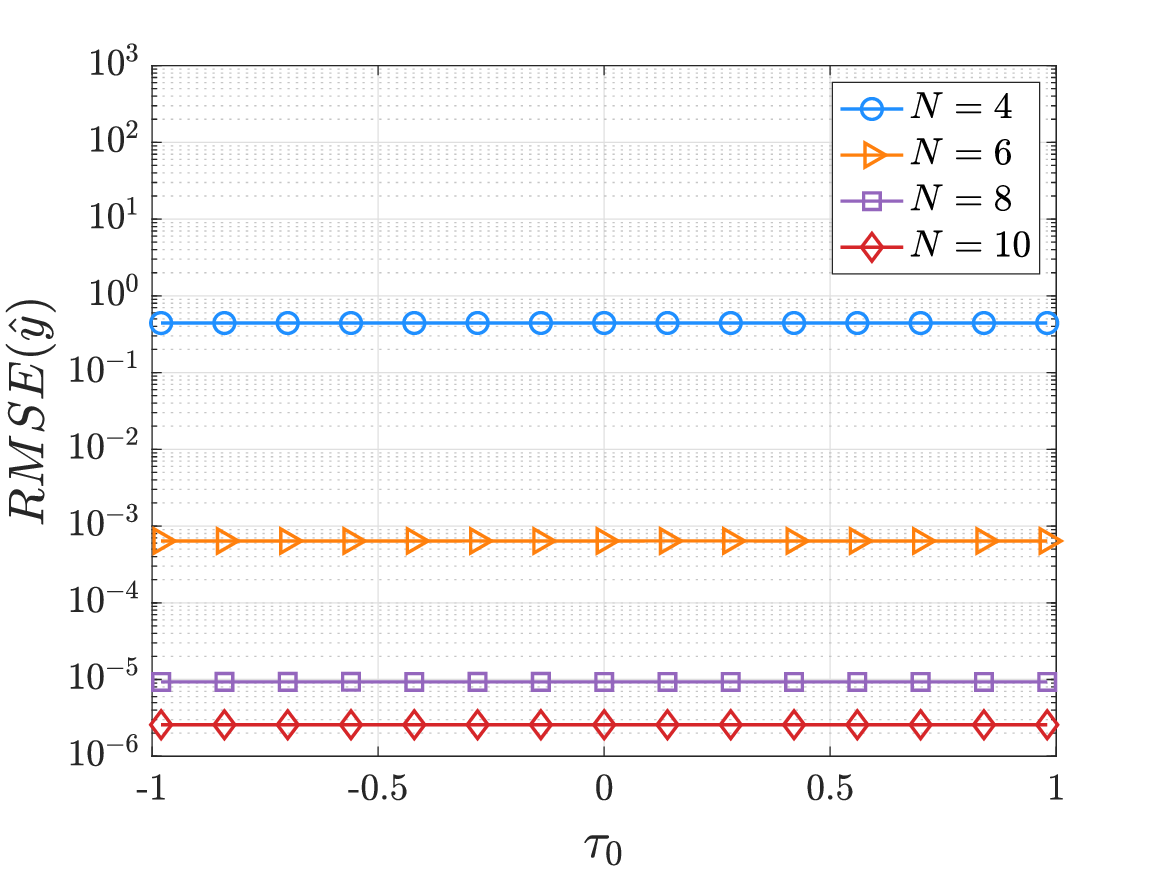}
    \caption{Root mean square error in the polynomial approximation of the state for Example 1 as a function of the location of the additional support point $\tau_{N+1}\in(-1,+1)$ when the method of Ref.~\cite{Garrido2023} (which is equivalent to the second integral form of Section \ref{section:second integral form}). The corresponding RMSE for the LGL-I method gives similar behavior and is not shown.}\label{fig: 1Divp taustar position}
\end{figure}

\subsubsection{Accuracy Analysis Using Multiple Intervals}

While Section \ref{sect:example-1-single-interval} demonstrates the accuracy of the LGL-I method as a function of $N$ using a single interval, in practice using a high-degree polynomial approximation is computationally intractible.  First, using a high-degree polynomial leads to overfitting of the data and causes numerical instability because small errors in the data can cause large variations in the polynomial values \cite{Berrut2004}.  Furthermore, as the polynomial degree increases, the main diagonal blocks in the NLP constraint Jacobian become increasingly dense which leads to computational intractability as the polynomial degree increases \cite{Darby2010}.

On the other hand, dividing the optimal control problem in to multiple interval and using a low-degree polynomial in each interval (see Appendix~\ref{sect: appendix}), the NLP constraint Jacobian becomes increasingly sparse as $N$ increases.  Furthermore, when using the derivative-like form of LGL collocation as described in Section \ref{sect:LGL deriv} (which as shown, is equivalent to the integral form described in Section \ref{sect:LGL int}), not only does the NLP constraint Jacobian become increasingly sparse, so too does the Hessian of the NLP Lagrangian.  This sparse structure increases significantly the ability to solve the NLP both accurately and efficiently using advanced NLP solvers \cite{Gill2002,Biegler2008}.  

\begin{figure}[t!]
    \centering
    \begin{subfigure}{0.5\textwidth}
        \centering
        \includegraphics[width=\linewidth]{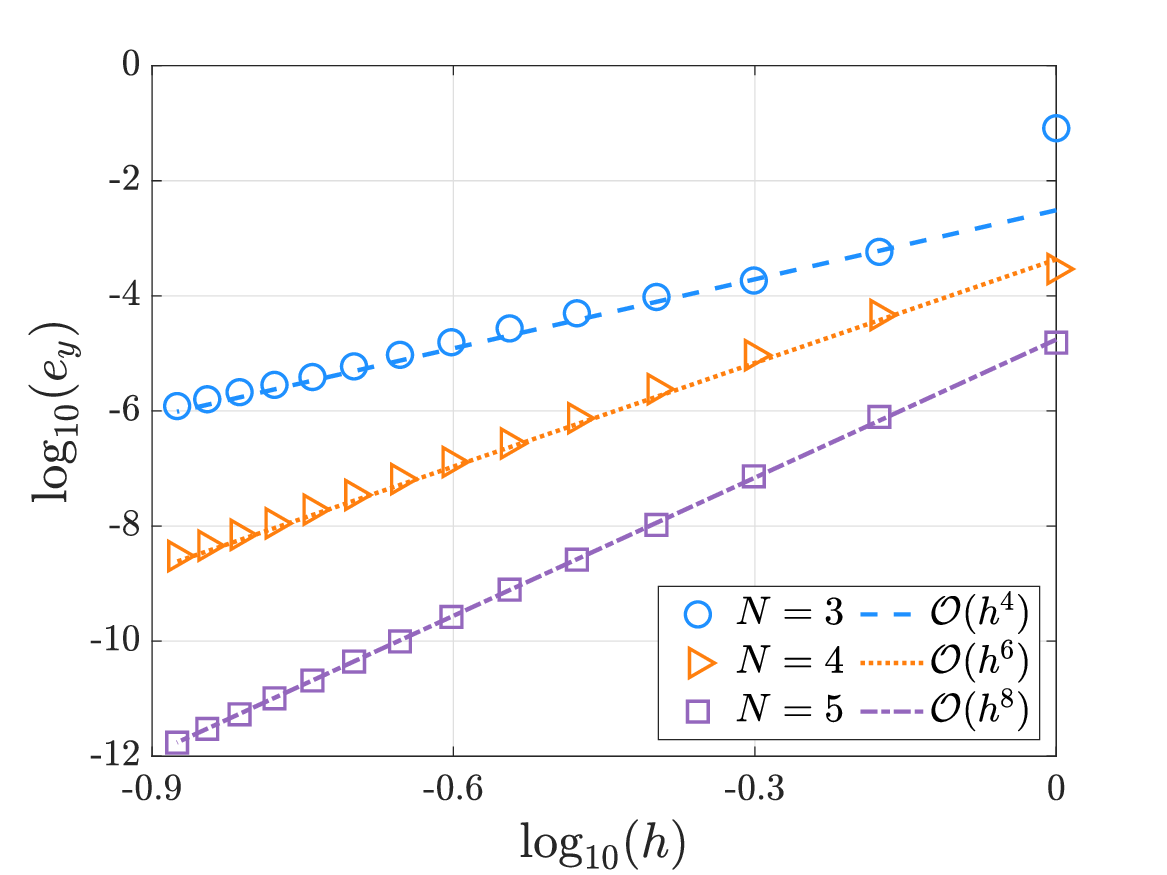}
        \caption{State error.}\label{subfig: state-error-example-1}
    \end{subfigure}%
    \begin{subfigure}{0.5\textwidth}
        \centering
        \includegraphics[width=\linewidth]{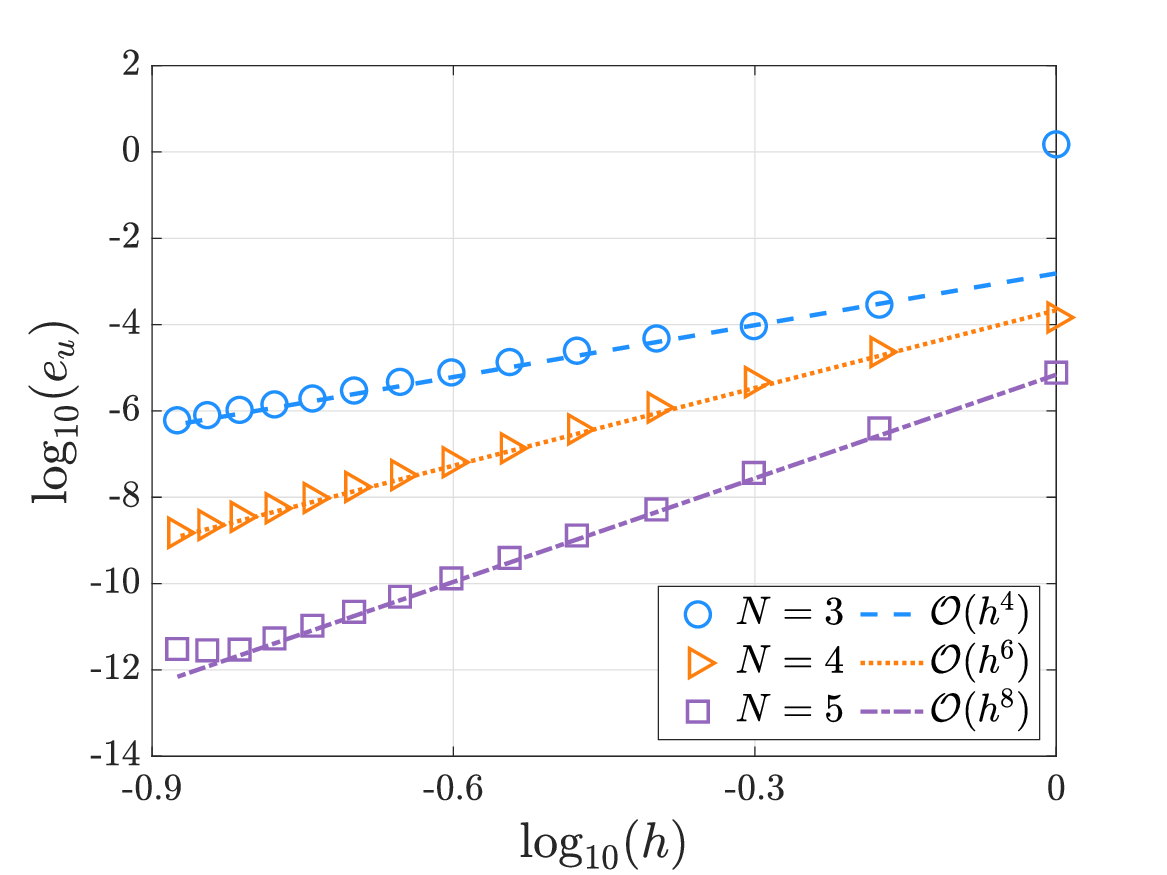}
        \caption{Control error.}\label{subfig: control-errors-example-1}
    \end{subfigure}\\
    \begin{subfigure}{0.5\textwidth}
        \centering
        \includegraphics[width=\linewidth]{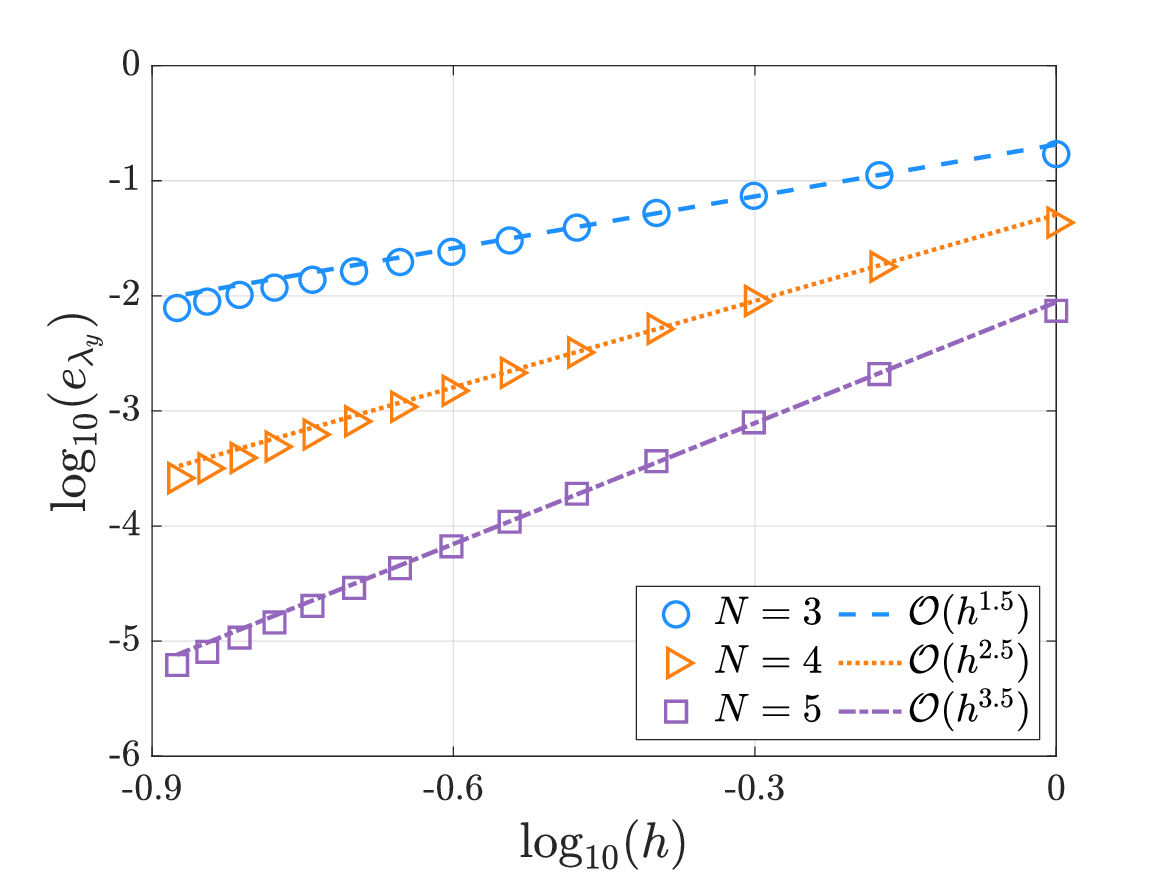}
        \caption{Costate error.}\label{subfig: costate-error-example-1}
    \end{subfigure}%
    \begin{subfigure}{0.5\textwidth}
        \centering
        \includegraphics[width=\linewidth]{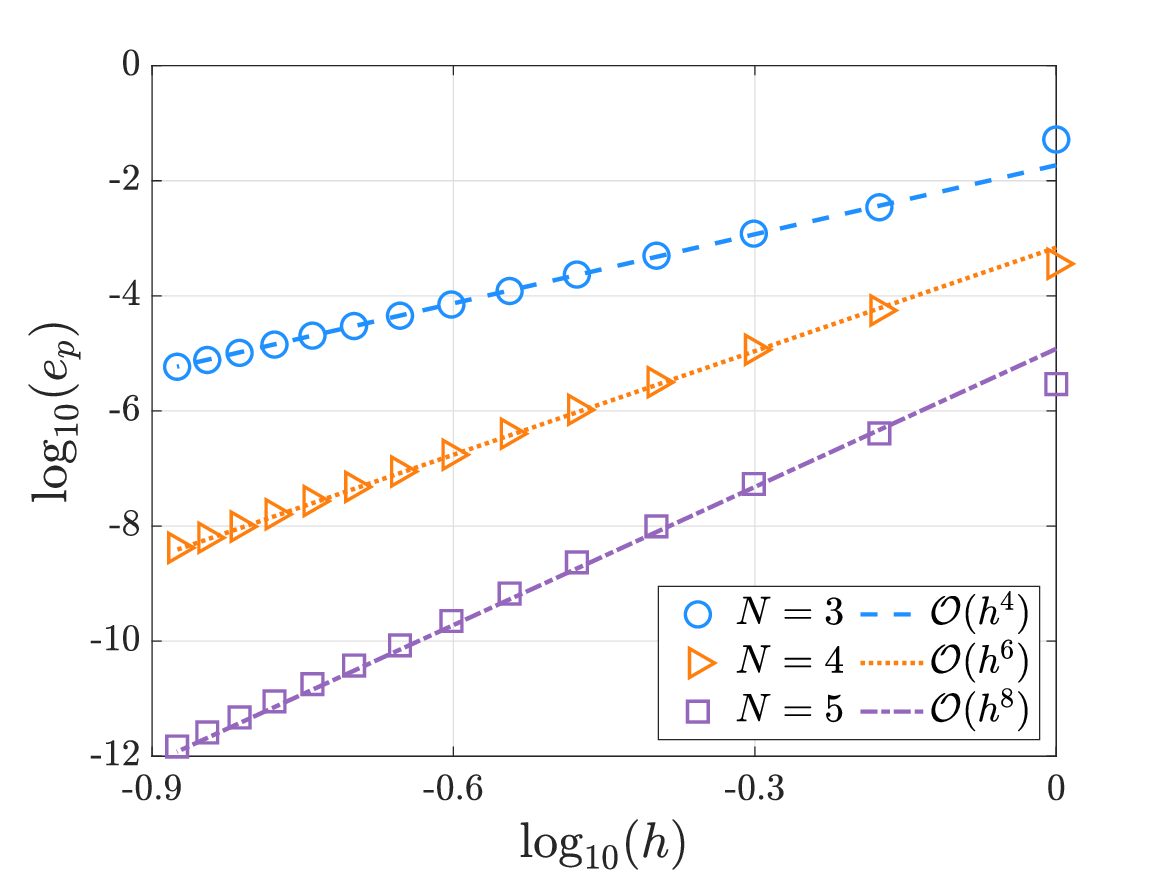}
        \caption{Superconvergent costate error.}\label{fig: costate-superconvergence-error-example-1}
    \end{subfigure}
    \caption{Solution error for Example 1 at the mesh points as a function of interval size for $N=\{3,4,5\}$.}\label{fig: 1Divp errors by h}
\end{figure}

Consider now a multiple-interval form of the LGL-I method as shown in Appendix~\ref{sect: appendix}.  It is known for each interval in a multiple-interval form that the LGL-I method satisfies the conditions of an implicit Runge-Kutta scheme \cite{Hager2000,Hager2025} when one uses the matrix $\m{A}$ as defined in Eq.~\eqref{Axelsson integration matrix elements}. Furthermore, the convergence rates for a Runge-Kutta scheme in optimal control can be determined from the order conditions given in Ref.~\cite{Hager2000} for orders up to four, while the conditions for higher-order schemes are derived in Ref.~\cite{Bonnans2006} and tabulated in Ref.~\cite{Varin2005}. Based on these order conditions, for a given mesh spacing $h$, the multiple-interval LGL discretization has a superconvergence rate  \cite{Hager2025} that is $\C{O}(h^{2N-2})$ for $N=(2,3,4,5)$ \cite{Hager2000} at the mesh points.  

To analyze the observed convergence rates of the multiple-interval LGL-I method, the example in Eq.~\eqref{example-1} is solved a second time using the LGL-I method for a uniformly spaced mesh of $K$ intervals with $N=\{3,4,5\}$ collocation points per interval. Figures~\ref{subfig: state-error-example-1} and \ref{subfig: control-errors-example-1} show the state and control for this example, while Fig.~\ref{subfig: costate-error-example-1} shows the costate obtained in the method of this paper.  It is seen that the state and control adhere to the superconvergence rate of $h^{2N-2}$, while the costate converges much more slowly than the state and control.  Note, however, that when solving the discrete approximation using multiple intervals, a second costate approximation, denoted $\m{p}_k$, exists at each mesh point $k\in(1,\ldots,K)$, and it is this costate that exhibits the superconvergence property \cite{Hager2025}. After solving the multiple-interval optimal control problem using the LGL-I method developed in this paper, this second costate $\m{p}_k$ is obtained by solving the following system of linear equations:
\begin{equation}\label{costate superconvergence eq}
    \begin{aligned}
    \m{p}_k &= \m{p}_{k+1} +
(h_{k}/2)\sum_{j=1}^N w_j
\nabla_x H(\m{X}_{j}^{(k)}, \m{U}_{j}^{(k)}, \m{q}_j^{(k)}), \quad
\m{p}_K = \nabla \Phi(\m{X}_N^{(K)}), \\
\m{q}_{i}^{(k)} &= \m{p}_{k+1} +
(h_k/2) \sum_{j=1}^N \left( \frac{w_j A_{ji}}{w_i} \right)
\nabla_x H(\m{X}_{j}^{(k)}, \m{U}_{j}^{(k)}, \m{q}_j^{(k)}), \quad
1 \le i \le N,    
    \end{aligned}
\end{equation}
where $h_k$ is the length of the $k^{\tx{th}}$ mesh interval and $H(\m{x}, \m{u}, \m{p}) = \m{p}\T \m{f}(\m{x}, \m{u})$ is the Hamiltonian (see Section~3.5 of Ref.~\cite{Hager2025} for further details).  The error in the costate approximation of Eq.~\eqref{costate superconvergence eq} is shown in Fig.~\ref{fig: costate-superconvergence-error-example-1} and it is seen that the convergence rate for $\m{p}_k$ is the same as the convergence rate for the state and control.

\subsection{Example 2: Orbit Raising Optimal Control Problem}\label{subsect: ex2}

\begin{figure}[b!]
    \centering
    \begin{subfigure}{0.45\textwidth}
        \centering
        \includegraphics[width=\linewidth]{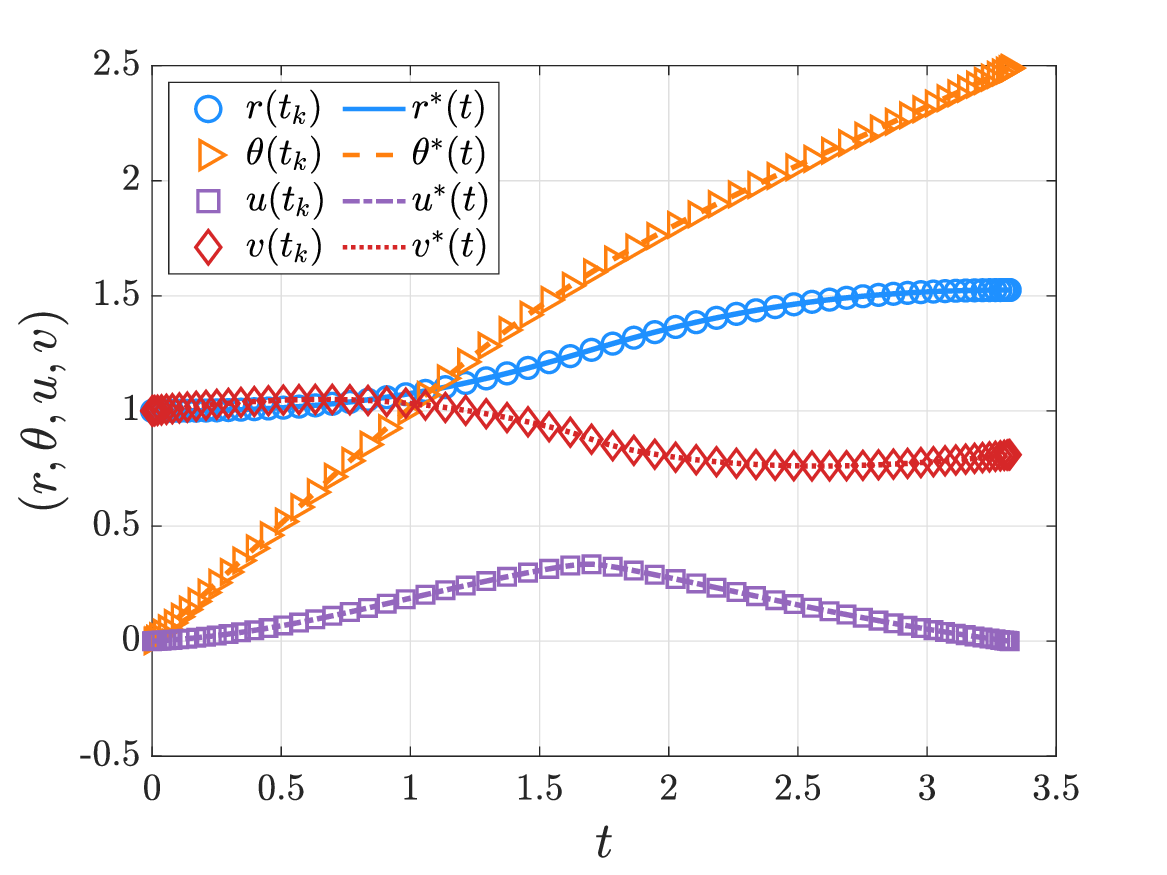}
        \caption{State solution.}\label{subfig: OrbitRaising states}
    \end{subfigure}%
    \begin{subfigure}{0.45\textwidth}
        \centering
        \includegraphics[width=\linewidth]{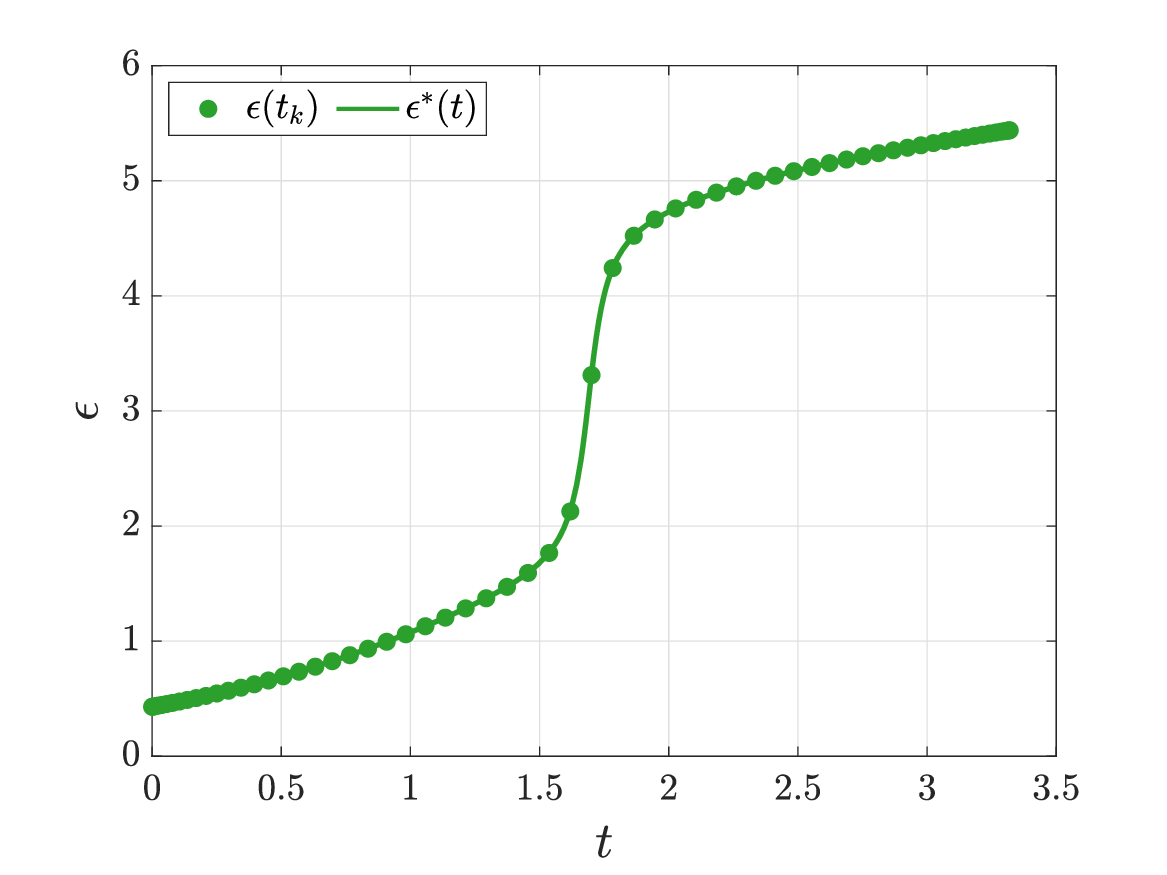}
        \caption{Control Solution.}\label{subfig: OrbitRaising control}
    \end{subfigure}\\
    \begin{subfigure}{0.45\textwidth}
        \centering
        \includegraphics[width=\linewidth]{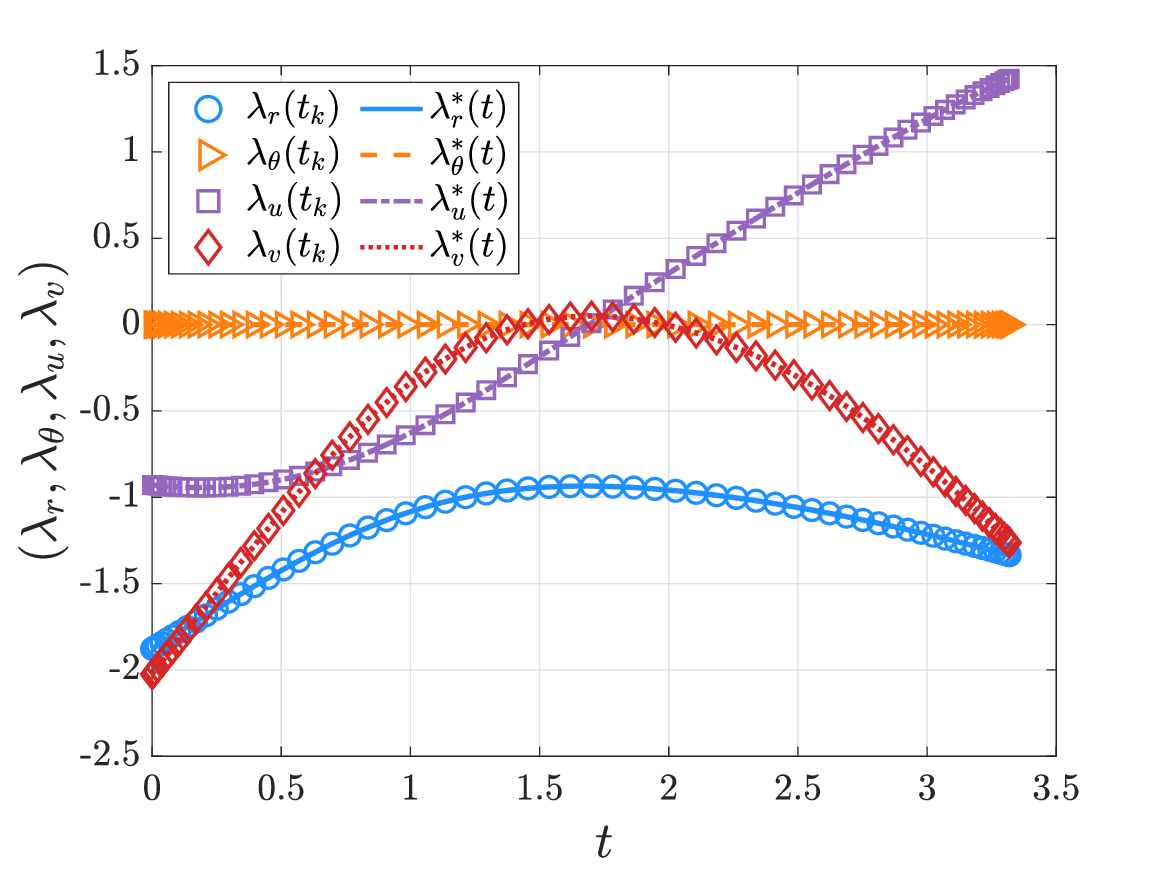}
        \caption{Costate solution.}\label{subfig: OrbitRaising costate}
    \end{subfigure}
    \caption{Solution for Example 2 using LGL-I collocation with $N=64$ LGL points alongside solution obtained using indirect shooting.}\label{fig: OrbitRaising Solution}
\end{figure}

Consider the following optimal control problem \cite{Bryson1975,Fahroo2001}.  
\begin{equation}\label{orbit raising problem}
    \tx{minimize}~ -r(t_f) \quad \tx{subject to}~ \begin{cases}\ds
        \ddt{r}{t} = u, & r(0)=1, \\[5pt] \ds
        \ddt{\theta}{t} = \frac{v}{r}, & \theta(0) = 0,\\[5pt] \ds
        \ddt{u}{t} = \frac{v^2}{r} - \frac{1}{r^2} + A(t)\sin\epsilon, & u(0)=0, ~u(t_f) = 0,\\[5pt] \ds
        \ddt{v}{t} = -\frac{uv}{r} + A(t)\cos\epsilon, & v(0) = 1, ~v(t_f) = \sqrt{1/r(t_f)},
    \end{cases} 
\end{equation}
where $A(t) = T/(m_0 - |\dot{m}|t)$, $T$ is the thrust magnitude, $m_0$ is the initial mass, and $\dot{m}$ is the constant fuel consumption rate. The normalized constants are given by $t_f=3.32$, $T=0.1405$, $m_0=1.0$, and $|\dot{m}|=0.0749$. First, a baseline solution to the optimal control problem in Eq.~\eqref{orbit raising problem} was obtained using indirect multiple-shooting. Specifically, the Hamiltonian boundary value problem (HBVP), derived from the variational first-order optimality conditions, was solved to a tolerance of $10^{-12}$ using the MATLAB functions {\tt fsolve} (for root-finding) and the differential equation solver {\tt ode113} (for simulation of the state-costate Hamiltonian system). Second, the optimal control problem given in Eq.~\eqref{orbit raising problem} was discretized into an NLP using the LGL-I method of this paper and a second NLP using the LGL method of Ref.~\cite{Fahroo2001}.  Both NLPs were solved using the NLP solver IPOPT with a convergence tolerance of $10^{-8}$.

\subsubsection{Accuracy Analysis Using a Single Interval\label{sect:example-2-single-interval}}
This example was previously solved in \cite{Fahroo2001} using the LGL collocation method with $N=64$ collocation points. 
The initial guess was generated by numerically integrating the equations of motion given in Eq.~\eqref{orbit raising problem} with a constant control $\epsilon(t)=0.001$, as described in \cite{Fahroo2001}. The state, control, and costate solutions obtained with the LGL-I method are shown in Fig.~\ref{fig: OrbitRaising Solution}, with the baseline solution generated via indirect shooting denoted by the superscript asterisk, e.g. $r^*(t)$. By visual inspection, the LGL-I collocation solution is in close agreement with the baseline solution. 

\begin{figure}[htb!]
    \centering
    \includegraphics[width=\linewidth]{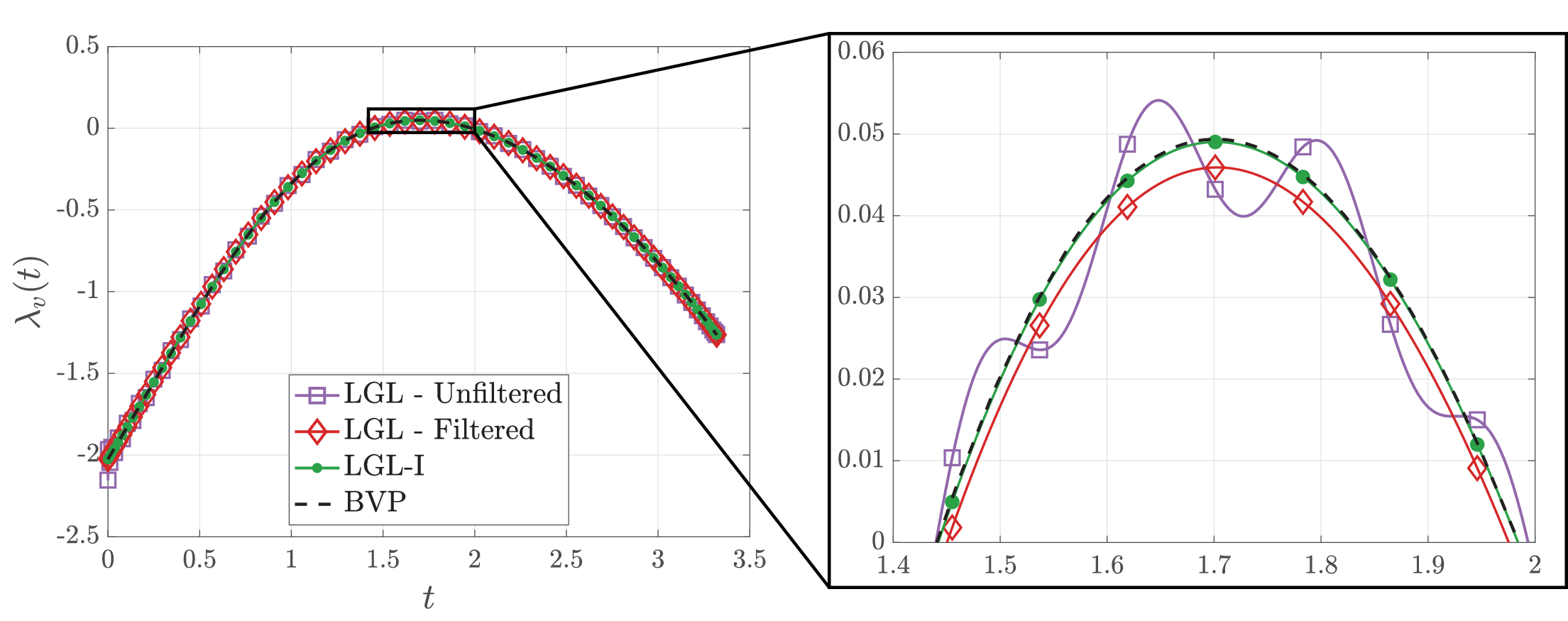}
    \caption{Comparison between $\lambda_v$ for Example 2 using different LGL collocation methods, with and without data filtering.}\label{fig: OrbitRaising lambdav Solution}
\end{figure}

As in Example 1, Example 2 demonstrates the same oscillatory behavior in the costate approximation obtained using the LGL\cite{Fahroo2001} method, but this oscillatory behavior is not present when using the LGL-I method. Figure~\ref{fig: OrbitRaising lambdav Solution} shows the costate approximation for $\lambda_v$. The raw costate solution from implementing the LGL\cite{Fahroo2001} method is labeled unfiltered, while the filtered costate estimate, which matches the results found in \cite{Fahroo2001}, is obtained by applying the filter shown in Appendix~\ref{sect: appendix filter} to the unfiltered costate. The LGL-I costate and filtered LGL\cite{Fahroo2001} costate are similar in trend, with the LGL-I costate approximating the baseline solution more closely. Thus, a key observation is that the LGL-I collocation method developed in this work achieves a higher accuracy costate approximation than LGL\cite{Fahroo2001} collocation and does so without the need for post-processing. Similar behavior was observed in all four costate components. The maximum absolute errors in the state, control, and costate components, defined similarly to those in Eq.~\eqref{abs error}, are provided in Table~\ref{tab:error comparison}. Because this example does not have an analytic solution, the error is approximated by referencing the baseline solution obtained using indirect multiple-shooting. The state and control errors obtained using the LGL-I collocation method are very similar in magnitude to those of the LGL\cite{Fahroo2001} method. It is observed, however, that the LGL-I method shows improved costate accuracy of at least one order of magnitude compared with the LGL\cite{Fahroo2001} method. 

\begin{table}[h!] 
\renewcommand{\arraystretch}{1.3} 
\setlength{\tabcolsep}{3pt}
\caption{Solution errors for Example 2 using a single interval.} 
\label{tab:error comparison}
\centering
\begin{tabular}{c | >{\centering\arraybackslash}p{2cm}  >{\centering\arraybackslash}p{2cm}  >{\centering\arraybackslash}p{2cm}}
\hline\hline
 & $\ds\max_{z\in\{r,\theta,u,v\}}e_z$ & $e_{\epsilon}$ & $\ds\max_{z\in\{r,\theta,u,v\}}e_{\lambda_z}$ \\
\hline
LGL-I & $4.9\scinot{-4}$ & $9.4\scinot{-3}$ & $1.1\scinot{-3}$ \\
LGL - Unfiltered & $9.5\scinot{-4}$ & $7.7\scinot{-2}$ & $3.0\scinot{0}$ \\
LGL - Filtered & --- & --- & $4.5\scinot{-2}$ \\
\hline\hline
\end{tabular}
\end{table}

\subsubsection{Accuracy Analysis Using Multiple Intervals}

As noted in Section~\ref{subsect: ex2}, it is uncommon to use such a high-degree global polynomial in practice to approximate the state. If Example 2 is solved again but using a multiple-interval form of the LGL-I method with 20 evenly spaced mesh intervals and $N=6$ collocation points in each interval, the maximum absolute state, control, and costate errors are $4.9\scinot{-5}$, $2.4\scinot{-4}$, and $2.2\scinot{-5}$, respectively. These multiple-interval results demonstrate an accuracy improvement of at least an order of magnitude when compared with the earlier results using a single interval. 

\section{Conclusions}\label{sect:conclusion}

A novel Legendre-Gauss-Lobatto direct orthogonal collocation method has been presented. The method approximates the right-hand side vector field of the differential equations using a basis of Lagrange polynomials and then evaluates the integrals of these Lagrange polynomials exactly using Legendre-Gauss-Lobatto quadrature. The resulting integral form of the presented LGL collocation method can then be cast into an equivalent derivative-like form. The transformed adjoint systems and corresponding costate estimates for both forms have been derived, with the transformed adjoint systems shown to be equivalent and full-rank. A second integral form was then derived by augmenting the first integral form with an integral to an additional noncollocated point, and it was shown that this second integral form is equivalent to a previously developed full-rank LGL derivative form. It was also shown that the inclusion of this noncollocated point is superfluous and its particular placement on the domain is arbitrary.  Finally, the Legendre-Gauss-Lobatto collocation method developed in this paper was applied to two examples where convergence of the state, control, and costate was demonstrated numerically.  

\appendix

\section{Digital Filter Implementation}\label{sect: appendix filter}

\begin{lstlisting}[language=Matlab]
function csOut = OptimalPrimeLPMcsFilt(lglPoints, csIn)
%%%%%%%%%%%%%%%%%%%%%%%%%%%%%%%%%%%%%%%%%%%%%%%%%%%%%%%%%%%%%%%%%%%%%%%%%%%%%%
%
% OptimalPrimeLPMcsFilt.m
%
% Filters the LGL (2001) costate. Digital filter designed by Tom P.
% Thorvaldsen of The Charles Stark Draper Laboratory, Circa 2000. The steps
% of the algorithm were communicated by Tom P. Thorvaldsen to Anil V. Rao 
% in 2009. These steps were then implemented by Dr. Michael A. Patterson in  
% 2009, and implemented by Dr. Michael A. Patterson in the code shown below.  
% The code shown below is used by permission of Michael A. Patterson.
%
%%%%%%%%%%%%%%%%%%%%%%%%%%%%%%%%%%%%%%%%%%%%%%%%%%%%%%%%%%%%%%%%%%%%%%%%%%%%%%

%number of endpoints removed from the data,
numCut = 2;

csSection = csIn(numCut+1:end-numCut,:);

% filter constants needed for 3 point filter
B = [0.25, 0.5, 0.25];
A = 1;
% filter the section data
% NOTE: filter shifts data
filtOut = filter(B,A,csSection);

% shift filter data back (X)
% match lgl points to filtered data (Y)
YY = filtOut(3:end,:);
XX = lglPoints(numCut+2:end-(numCut+1),1);

csOutI = interp1(XX,YY,lglPoints(1:numCut+1,:),'linear','extrap');
csOutF = interp1(XX,YY,lglPoints(end-numCut:end,:),'linear','extrap');

csOut = [csOutI; YY; csOutF];
\end{lstlisting}

\section{Multiple-Interval LGL Collocation for Optimal Control}\label{sect: appendix}
Without loss of generality, consider the following form of an optimal control problem. Determine the state $\m{x}(t)\in\R^{n_x}$, control $\m{u}(t)\in\R^{n_u}$, initial time $t_0$, and final time $t_f$ that minimize the objective functional subject to dynamic constraints and boundary conditions. That is,
\begin{equation}
    \tx{minimize } \Phi(\m{x}(t_0),t_0,\m{x}(t_f),t_f)    \tx{ subject to } \left\{ \begin{aligned}
            \ddt{\m{x}(t)}{t} &= \m{f}(\m{x}(t),\m{u}(t)),\\
            \m{0} &= \m{b}(\m{x}(t_0),t_0,\m{x}(t_f),t_f),
        \end{aligned}\right.
\end{equation}
where the functions $\Phi$, $\m{f}$, and $\m{b}$ are defined by the mappings 
\begin{align}
    \Phi & : \R^{n_x} \times \R \times \R^{n_x} \times \R \rightarrow \R, \\
    \m{f} & : \R^{n_x} \times \R^{n_u} \rightarrow \R^{n_x},\\
    \m{b} &: \R^{n_x} \times \R \times \R^{n_x} \times \R \rightarrow \R^{n_b}.
\end{align}
The multiple-interval integral form of the LGL collocation method is implemented as follows. First, suppose the domain $t\in[t_0,t_f]$ is transformed to $\tau\in[-1,+1]$ using Eq.~\eqref{affine trans}. Next, suppose the domain $\tau\in[-1,+1]$ is segmented into a $K$-interval mesh with $K+1$ mesh points $(T_0,\ldots,T_K)$ such that $T_0=-1$, $T_K=+1$, $T_0<T_1<\cdots<T_K$ (that is, the mesh points are strictly monotonically increasing), and $\C{I}_k=[T_{k-1},T_k]$ denotes the $k^{\tx{th}}$ mesh interval. The mesh intervals have the property that $\bigcup_{k=1}^{K} \C{I}_k = [-1,+1]$ and $\C{I}_k \cap \C{I}_{k+1} = \{ T_k \},~(k=1,\ldots,K-1)$. The differential equation vector field is then approximated in each mesh interval as
\begin{equation}\label{multi f approx}
    \hat{\m{f}}\intk(\tau)= \sum_{j=1}^{N_k} \m{F}_j\intk L_j\intk(\tau), \quad \m{F}\intk_j = \m{f}\left(\m{X}\intk_j,\m{U}\intk_j\right), \quad (k=1,\ldots,K),
\end{equation}
where
\begin{equation}
    L_j\intk(\tau) = \prod_{\substack{i = 1 \\ i \neq j}}^{N_k} \frac{\tau - \tau\intk_i}{\tau\intk_j - \tau\intk_i}, \quad (j=1,\ldots,N_k),
\end{equation}
are the Lagrange polynomials of degree $N_k-1$ with support points at the $N_k$ LGL points $(\tau\intk_1,\ldots,\tau\intk_{N_k})$ in interval $k$, and $(\m{X}\intk, \m{U}\intk)$ are the approximations of the state and the control at each of the $N_k$ LGL points in interval $k\in\{1,\ldots,K\}$. The computational domain of each interval is $\tau\intk\in[T_{k-1},T_k],~(k=1,\ldots,K)$. Following the same steps of Section~\ref{sect:LGL int} in each mesh interval, the collocation constraints are constructed in integral form by integrating the polynomial vector field approximation of Eq.~\eqref{multi f approx} to each LGL point in a given interval. Thus, the multiple-interval analog of the NLP in Eq.~\eqref{integral NLP} is given by
\begin{equation}
    \begin{aligned}
        \tx{minimize} \quad & \Phi\left(\m{X}_{1}\intO,t_0,\m{X}_{N_K}\intK,t_f\right)  \\
        \tx{subject to} \quad & \left\{ \begin{aligned}
            \m{0} &=  \begin{bmatrix} \m{1} & -\m{I}_{N_k-1} \end{bmatrix} \m{X}\intk + \frac{t_f-t_0}{2}\tilde{\m{A}}\intk\m{F}\intk, \quad (k=1,\ldots,K),\\
            \m{0} &= \m{b}\left(\m{X}_{1}\intO,t_0,\m{X}_{N_K}\intK,t_f\right),\\
        \end{aligned}\right.
    \end{aligned}
\end{equation}
where $w_j\intk,~(j=1,\ldots,N_k)$, are the LGL quadrature weights in the $k^{\tx{th}}$ mesh interval, and $\tilde{\m{A}}\intk:=\m{A}\intk_{(2:N_k,:)}$, where 
\begin{equation}
    A\intk_{ij} := \int_{T_{k-1}}^{\tau\intk_i} L\intk_j(\tau)\dt{\tau}, \quad (i,j=1,\ldots,N_k),
\end{equation}
are the elements of the $N_k\times N_k$ integration matrix in the $k^{\tx{th}}$ mesh interval. Furthermore, continuity in the state across interior mesh points is enforced via the condition
\begin{equation}
    {\m{X}}\intk_{N_k} = {\m{X}}^{(k+1)}_1, \quad (k=1,\ldots,K-1).
\end{equation}


\section*{Acknowledgments}
The authors gratefully acknowledge support for this research from the U.S. National Science Foundation under the Graduate Research Fellowship Program and grant CMMI-2031213 and from the U.S. Office of Naval Research under grant N00014-22-1-2397. The authors also acknowledge Tom Thorvaldsen of The Charles Stark Draper Laboratory Inc. for providing the steps required for the implementation of the LGL costate filter.  Finally, the authors acknowledge Dr.~Michael A.~Patterson for providing the implementation of the costate filter provided by Tom Thorvaldsen as found in the Appendix of this paper.


\bibliographystyle{aiaa}

\begin{thebibliography}{10}
\newcommand{\enquote}[1]{``#1''}

\bibitem{Kirk2004}
Kirk, D.~E., {\em Optimal Control Theory: An Introduction\/}, Courier
  Corporation, 2004.

\bibitem{Biegler2008}
Biegler, L.~T. and Zavala, V.~M., \enquote{Large-scale nonlinear programming
  using {IPOPT}: An integrating framework for enterprise-wide dynamic
  optimization,} {\em Computers \& Chemical Engineering\/}, Vol.~33, No.~3,
  March 2009, pp.~575--582,
  \url{https://doi.org/10.1016/j.compchemeng.2008.08.006}.

\bibitem{Gill2002}
Gill, P.~E., Murray, W., and Saunders, M.~A., \enquote{{SNOPT}: An {SQP}
  Algorithm for Large-Scale Constrained Optimization,} {\em {SIAM} Review\/},
  Vol.~47, No.~1, Jan. 2005, pp.~99--131,
  \url{https://doi.org/10.1137/s0036144504446096}.

\bibitem{Biral2016}
Biral, F., Bertolazzi, E., and Bosetti, P., \enquote{Notes on Numerical Methods
  for Solving Optimal Control Problems,} {\em IEEJ Journal of Industry
  Applications\/}, Vol.~5, No.~2, March--April 2016, pp.~154--166,
  \url{https://doi.org/10.1541/ieejjia.5.154}.

\bibitem{Taheri2016}
Taheri, E., Kolmanovsky, I., and Atkins, E., \enquote{Enhanced Smoothing
  Technique for Indirect Optimization of Minimum-Fuel Low-Thrust Trajectories,}
  {\em Journal of Guidance, Control, and Dynamics\/}, Vol.~39, No.~11, Nov.
  2016, pp.~2500--2511, \url{https://doi.org/10.2514/1.g000379}.

\bibitem{Mall2022}
Mall, K. and Taheri, E., \enquote{Three-Degree-of-Freedom Hypersonic Reentry
  Trajectory Optimization Using an Advanced Indirect Method,} {\em Journal of
  Spacecraft and Rockets\/}, Vol.~59, No.~5, September--October 2022,
  pp.~1463--1474, \url{https://doi.org/10.2514/1.A34893}.

\bibitem{Bryson1975}
Bryson, A.~E. and Ho, Y.-C., {\em Applied Optimal Control: Optimization,
  Estimation, and Control\/}, Hemisphere Publishing Corporation, 1975.

\bibitem{Seywald1996}
Seywald, H. and Kumar, R.~R., \enquote{Method for Automatic Costate
  Calculation,} {\em Journal of Guidance, Control, and Dynamics\/}, Vol.~19,
  No.~6, November--December 1996, pp.~1252--1261,
  \url{https://doi.org/10.2514/3.21780}.

\bibitem{Gong2008b}
Gong, Q., Ross, I.~M., Kang, W., and Fahroo, F., \enquote{Connections Between
  the Covector Mapping Theorem and Convergence of Pseudospectral Methods for
  Optimal Control,} {\em Computational Optimization and Applications\/},
  Vol.~41, Dec. 2008, pp.~307--335,
  \url{https://doi.org/10.1007/s10589-007-9102-4}.

\bibitem{Garg2010}
Garg, D., Patterson, M., Hager, W.~W., Rao, A.~V., Benson, D.~A., and
  Huntington, G.~T., \enquote{A unified framework for the numerical solution of
  optimal control problems using pseudospectral methods,} {\em Automatica\/},
  Vol.~46, No.~11, Nov. 2010, pp.~1843--1851,
  \url{https://doi.org/10.1016/j.automatica.2010.06.048}.

\bibitem{Garg2011a}
Garg, D., Hager, W.~W., and Rao, A.~V., \enquote{Pseudospectral Methods for
  Solving Infinite-Horizon Optimal Control Problems,} {\em Automatica\/},
  Vol.~47, No.~4, April 2011, pp.~829--837,
  \url{https://doi.org/10.1016/j.automatica.2011.01.085}.

\bibitem{Garg2011b}
Garg, D., Patterson, M.~A., Francolin, C., Darby, C.~L., Huntington, G.~T.,
  Hager, W.~W., and Rao, A.~V., \enquote{Direct Trajectory Optimization and
  Costate Estimation of Finite-Horizon and Infinite-Horizon Optimal Control
  Problems using a {R}adau Pseudospectral Method,} {\em Computational
  Optimization and Applications\/}, Vol.~49, No.~2, June 2011, pp.~335--358,
  \url{https://doi.org/10.1007/s10589-009-9291-0}.

\bibitem{Darby2011sep}
Darby, C.~L., Garg, D., and Rao, A.~V., \enquote{Costate Estimation using
  Multiple-Interval Pseudospectral Methods,} {\em Journal of Spacecraft and
  Rockets\/}, Vol.~48, No.~5, Sept. 2011, pp.~856--866,
  \url{https://doi.org/10.2514/1.a32040}.

\bibitem{Betts1998a}
Betts, J.~T., \enquote{Survey of numerical methods for trajectory
  optimization,} {\em Journal of Guidance, Control, and Dynamics\/}, Vol.~21,
  No.~2, March--April 1998, pp.~193--207, \url{https://doi.org/10.2514/2.4231}.

\bibitem{RaoSurvey}
Rao, A.~V., \enquote{A survey of numerical methods for optimal control,} {\em
  Advances in the Astronautical Sciences\/}, Vol.~135, No.~1, 2009,
  pp.~497--528.

\bibitem{Dahlquist2003}
Dahlquist, G. and Bj{\"o}rck, {\AA}., {\em Numerical Methods\/}, Courier
  Corporation, 2003.

\bibitem{Hull1997}
Hull, D.~G., \enquote{Conversion of Optimal Control Problems into Parameter
  Optimization Problems,} {\em Journal of Guidance, Control, and Dynamics\/},
  Vol.~20, No.~1, January--February 1997, pp.~57--60,
  \url{https://doi.org/10.2514/2.4033}.

\bibitem{Betts2020}
Betts, J.~T., {\em Practical Methods for Optimal Control Using Nonlinear
  Programming\/}, SIAM Press, Philadelphia, Pennsylvania, 3rd ed., 2020.

\bibitem{Ascher1995}
Ascher, U.~M., Mattheij, R.~M., and Russell, R.~D., {\em Numerical solution of
  boundary value problems for ordinary differential equations\/}, SIAM, 1995.

\bibitem{Herman1996}
Herman, A.~L. and Conway, B.~A., \enquote{Direct Optimization Using Collocation
  Based on High-Order {G}auss-{L}obatto Quadrature Rules,} {\em Journal of
  Guidance, Control, and Dynamics\/}, Vol.~19, No.~3, May--June 1996,
  pp.~592--599, \url{https://doi.org/10.2514/3.21662}.

\bibitem{Betts1999}
Betts, J.~T. and Huffman, W.~P., \enquote{Exploiting Sparsity in the Direct
  Transcription Method for Optimal Control,} {\em Computational Optimization
  and Applications\/}, Vol.~14, No.~2, Sept. 1999, pp.~179--201,
  \url{https://doi.org/10.1023/A:1008739131724}.

\bibitem{Williams2009}
Williams, P., \enquote{{H}ermite-{L}egendre-{G}auss-{L}obatto Direct
  Transcription in Trajectory Optimization,} {\em Journal of Guidance, Control,
  and Dynamics\/}, Vol.~32, No.~4, July--August 2009, pp.~1392--1395,
  \url{https://doi.org/10.2514/1.42731}.

\bibitem{Liu2014}
Liu, Y., Zhu, H., Huang, X., and Zheng, G., \enquote{A {H}ermite-{L}obatto
  Pseudospectral Method for Optimal Control,} {\em Asian Journal of Control\/},
  Vol.~16, No.~5, Sept. 2014, pp.~1568--1575,
  \url{https://doi.org/10.1002/asjc.869}.

\bibitem{Pezent2024}
Pezent, J., {\em On the Implementation and Applications of a High Performance
  Trajectory Optimization Toolkit\/}, Ph.D. thesis, The University of Alabama,
  2024.

\bibitem{Benson2006}
Benson, D.~A., Huntington, G.~T., Thorvaldsen, T.~P., and Rao, A.~V.,
  \enquote{Direct trajectory optimization and costate estimation via an
  orthogonal collocation method,} {\em Journal of Guidance, Control, and
  Dynamics\/}, Vol.~29, No.~6, Nov. 2006, pp.~1435--1440,
  \url{https://doi.org/10.2514/1.20478}.

\bibitem{Elnagar1995}
Elnagar, G., Kazemi, M., and Razzaghi, M., \enquote{The pseudospectral Legendre
  method for discretizing optimal control problems,} {\em {IEEE} Transactions
  on Automatic Control\/}, Vol.~40, No.~10, Oct. 1995, pp.~1793--1796,
  \url{https://doi.org/10.1109/9.467672}.

\bibitem{Fahroo2001}
Fahroo, F. and Ross, I.~M., \enquote{Costate Estimation by a {L}egendre
  Pseudospectral Method,} {\em Journal of Guidance, Control, and Dynamics\/},
  Vol.~24, No.~2, March 2001, pp.~270--277,
  \url{https://doi.org/10.2514/2.4709}.

\bibitem{Fahroo2008}
Fahroo, F. and Ross, I.~M., \enquote{Pseudospectral Methods for
  Infinite-Horizon Nonlinear Optimal Control Problems,} {\em Journal of
  Guidance, Control, and Dynamics\/}, Vol.~31, No.~4, July--August 2008,
  pp.~927--936, \url{https://doi.org/10.2514/1.33117}.

\bibitem{Gong2008a}
Gong, Q., Fahroo, F., and Ross, I.~M., \enquote{Spectral Algorithm for
  Pseudospectral Methods in Optimal Control,} {\em Journal of Guidance,
  Control, and Dynamics\/}, Vol.~31, No.~3, May 2008, pp.~460--471,
  \url{https://doi.org/10.2514/1.32908}.

\bibitem{Garrido2023}
Garrido, J., Makarow, A., Sagliano, M., Seelbinder, D., and Theil, S.,
  \enquote{Costate Convergence with {L}egendre-{L}obatto Collocation for
  Trajectory Optimization,} Aug. 2023, \url{https://arxiv.org/abs/2307.14269}.

\bibitem{Hager2016}
Hager, W.~W., Hou, H., and Rao, A.~V., \enquote{Convergence rate for a {G}auss
  collocation method applied to unconstrained optimal control,} {\em Journal of
  Optimization Theory and Applications\/}, Vol.~169, No.~3, March 2016,
  pp.~801--824, \url{https://doi.org/10.1007/s10957-016-0929-7}.

\bibitem{Hager2018}
Hager, W.~W., Liu, J., Mohapatra, S., Rao, A.~V., and Wang, X.-S.,
  \enquote{Convergence rate for a Gauss collocation method applied to
  constrained optimal control,} {\em SIAM Journal on Control and
  Optimization\/}, Vol.~56, No.~2, 2018, pp.~1386--1411,
  \url{https://doi.org/10.1137/16M1096761}.

\bibitem{Hager2019b}
Hager, W.~W., Hou, H., Mohapatra, S., Rao, A.~V., and Wang, X.-S.,
  \enquote{Convergence Rate for a {R}adau hp Collocation Method Applied to
  Constrained Optimal Control,} {\em Computational Optimization and
  Applications\/}, Vol.~74, May 2019, pp.~275--314,
  \url{https://doi.org/10.1007/s10589-019-00100-1}.

\bibitem{Axelsson1964}
Axelsson, O., \enquote{Global integration of differential equations through
  {L}obatto quadrature,} {\em BIT Numerical Mathematics\/}, Vol.~4, June 1964,
  pp.~69--86, \url{https://doi.org/10.1007/BF01939850}.

\bibitem{Pontryagin1962}
Pontryagin, L.~S., Boltyanskii, V.~G., Gamkrelidze, R.~V., and Mishchenko,
  E.~F., {\em The Mathematical Theory of Optimal Processes\/}, Interscience
  Publishers, 1962.

\bibitem{Hager2025}
Hager, W.~W., {\em Computational Methods in Optimal Control: Theory and
  Practice\/}, SIAM, 2025.

\bibitem{Weinstein2017}
Weinstein, M.~J. and Rao, A.~V., \enquote{Algorithm 984: {AD}i{G}ator, a
  Toolbox for the Algorithmic Differentiation of Mathematical Functions in
  {MATLAB} Using Source Transformation via Operator Overloading,} {\em {ACM}
  Transactions on Mathematical Software\/}, Vol.~44, No.~2, Aug. 2017,
  pp.~1--25, \url{https://doi.org/10.1145/3104990}.

\bibitem{Berrut2004}
Berrut, J.-P. and Trefethen, L.~N., \enquote{Barycentric Lagrange
  Interpolation,} {\em SIAM Rev\/}, Vol.~46, 2004, pp.~501--517,
  \url{https://doi.org/10.1137/S0036144502417715}.

\bibitem{Darby2010}
Darby, C.~L., Hager, W.~W., and Rao, A.~V., \enquote{An hp-adaptive
  pseudospectral method for solving optimal control problems,} {\em Optimal
  Control Applications and Methods\/}, Vol.~32, No.~4, Aug. 2010, pp.~476--502,
  \url{https://doi.org/10.1002/oca.957}.

\bibitem{Hager2000}
Hager, W.~W., \enquote{Runge-{K}utta Methods in Optimal Control and the
  Transformed Adjoint System,} {\em Numerische Mathematik\/}, Vol.~87, No.~2,
  December 2000, pp.~247–282, \url{https://doi.org/10.1007/s002110000178}.

\bibitem{Bonnans2006}
Bonnans, J.~F. and Laurent-Varin, J., \enquote{Computation of Order Conditions
  for Symplectic Partitioned {R}unge-{K}utta Schemes with Application to
  Optimal Control,} {\em Numerische Mathematik\/}, Vol.~103, No.~1, March 2006,
  \url{https://doi.org/10.1007/s00211-005-0661-y}.

\bibitem{Varin2005}
Laurent-Varin, J., {\em Calcul de Trajectoires Optimales de Lanceurs Spatiaux
  R{\'e}utilisables par une m{\'e}thode de Point Int{\'e}rieur\/}, Ph.D.
  thesis, l{\'E}cole Polytechnique, Palaiseau, France, 2005.

\end{thebibliography}

\end{document}